\documentclass[11pt,letterpaper,english]{article}

\usepackage[T1]{fontenc}
\usepackage[utf8]{inputenc}
\usepackage[english]{babel}
\usepackage[margin=1in]{geometry}
\usepackage{authblk}
\usepackage{macros}

\tikzstyle{block} = [draw, fill=white, rectangle,
    minimum height=3em, minimum width=6em]
\tikzstyle{sum} = [draw, fill=white, circle, node distance=1cm]
\tikzstyle{input} = [coordinate]
\tikzstyle{output} = [coordinate]
\tikzstyle{pinstyle} = [pin edge={to-,thin,black}]
\tikzset{roads/.style={line width=0.2cm}}

\usepackage{times}

\begin{document}

\title{Random directions stochastic approximation with deterministic perturbations}

\author[1]{Prashanth L A\thanks{prashla@cse.iitm.ac.in}}
\author[2]{Shalabh Bhatnagar\thanks{shalabh@iisc.ac.in}}
\author[1]{Nirav Bhavsar\thanks{cs17s016@smail.iitm.ac.in}}
\author[3]{Michael Fu\thanks{mfu@isr.umd.edu}}
\author[4]{Steven I. Marcus\thanks{marcus@umd.edu}}

\affil[1]{\small Department of Computer Science and Engineering, Indian Institute of Technology Madras, Chennai}
\affil[2]{\small Department of Computer Science and Automation and the Robert Bosch Centre for Cyber Physical Systems,
Indian Institute of Science, Bangalore}
\affil[3]{\small Robert H. Smith School of Business \& Institute for Systems Research,
University of Maryland, College Park, Maryland}
\affil[3]{\small Department of Electrical and Computer Engineering \& Institute for Systems Research,
University of Maryland, College Park, Maryland}

\renewcommand\Authands{ and }

\date{}
\maketitle

\begin{abstract}
We introduce deterministic perturbation schemes for the recently proposed random directions stochastic approximation (RDSA) \cite{prashanth2017rdsa},
and propose new first-order and second-order algorithms. 
In the latter case, these are the first second-order algorithms to incorporate deterministic perturbations. 
We show that the gradient and/or Hessian estimates in the resulting algorithms with deterministic perturbations are asymptotically unbiased, so that the algorithms are provably convergent. 
Furthermore, we derive convergence rates to establish the superiority
of the first-order and second-order algorithms, for the special case of a convex and quadratic optimization problem, respectively. 
Numerical experiments are used to validate the theoretical results.
\end{abstract}



\section{Introduction}
\label{sec:intro}
We consider the following  problem:
\begin{align}
\mbox{Find } x^* = \arg\min_{x \in \R^N} f(x). \label{eq:pb}
\end{align}
%
We operate in a setting in which the analytical form of the objective function $f$ is not known, but noisy measurements of the function 
can be obtained. Furthermore, noisy estimates of the objective function gradient are not directly available, so the function gradient needs 
to be estimated using the aforementioned noisy measurements. Robbins and Monro \cite{rm}
developed an incremental-update algorithm that estimates the zeros of the function $f$ when only its noisy measurements are available. 
This algorithm has found applications in several engineering domains such as signal processing, manufacturing, communication networks, 
autonomous systems, vehicular traffic networks, etc., where it is often used to find either
(a) the fixed points of a certain function or (b) the optima of a certain objective given noisy function measurements. 

The earliest gradient search algorithm in this setting is the Kiefer-Wolfowitz \cite{kw} procedure. This, however, requires $2N$ function 
measurements when the parameter dimension is $N$. Katkovnik and Kulchitsky \cite{katkul, rubinstein} proposed a random search technique that 
became known as the smoothed functional (SF) algorithm. The key idea here is that the convolution
of the objective function gradient with a multivariate Gaussian PDF is seen via an integration-by-parts argument as the convolution of 
the objective function itself with a scaled multivariate
Gaussian. Thus, a single noisy function measurement at a perturbed value of the parameter update, perturbed using a multivariate Gaussian, 
is sufficient to obtain an estimate of the full gradient. 
This results in a one-measurement estimator that however has high bias. A balanced two-sided estimator of the gradient 
(requiring two function measurements) that has significantly lower bias 
than the one-measurement SF estimator was proposed in \cite{stybtang}, see also \cite{chin1997comparative} for comparisons of the one-measurement and two-measurement SF algorithms.

Random directions stochastic approximation (RDSA) \cite{kushcla} is another gradient search procedure, in which the perturbation variables 
are considered to be uniformly distributed 
over the surface of the unit sphere in $\R^N$. Obtaining these perturbation random variables is, however, computationally expensive, particularly 
when the dimension $N$ is large. In a landmark paper, Spall \cite{spall1992multivariate} introduced the 
simultaneous perturbation stochastic approximation algorithm, a random search technique that estimates the gradient using random perturbations that are independent, symmetric, zero-mean and satisfying an inverse moment bound. The most commonly used and studied class
of perturbations within this category are those that are independent, symmetric, $\pm 1$-valued, Bernoulli random variables. 
This algorithm (the standard SPSA, as it is known), requires two function measurements at each update step, and became popular because of its
computational simplicity, as well as the convergence and rate guarantees that it provides.
In another paper \cite{spall1997one}, Spall presented a one-measurement counterpart of SPSA. This algorithm, however, does not show good performance,
as it suffers from a large bias in its gradient estimates. Bhatnagar et al. \cite{bhatfumarcwang} presented certain deterministic perturbation variants of
SPSA. Here two constructions for the perturbation variates were proposed, of which, a construction based on Hadamard matrices is seen to show
remarkable improvements in the empirical performance of one-measurement SPSA.

Adaptive Newton-type schemes that estimate the Hessian using noisy objective function measurements, as with the gradient, have also
gathered considerable attention over the years. The earliest such scheme, due to Fabian \cite{fabian}, estimated the Hessian using
finite-difference estimates and required $O(N^2)$ samples of the objective function at each update epoch. Spall \cite{spall2000adaptive},
presented a simultaneous perturbation estimate of the Hessian that was based on four noisy function measurements. Two of these
measurements also estimate the gradient. In the case when noisy gradient measurements are directly available, he also presented
a Newton scheme requiring three measurements. Bhatnagar \cite{bhatnagar2005} presented three additional algorithms that estimate the Hessian
as well as the gradient, using simultaneous perturbation estimates. In the process, new gradient and Hessian SPSA estimators were developed.  
Bhatnagar and Prashanth \cite{bhatnagar-prashanth2015} presented a balanced estimator of the Hessian using three function measurements.
This paper also presented two algorithms, one of which estimated the inverse Hessian using a recursive procedure based on the
Sherman-Morrison-Woodbury lemma, while the other did not require one to compute or estimate the inverse Hessian at each step.
It was shown nonetheless that the asymptotic behaviour of the latter algorithm is analogous to a Newton algorithm that would involve
a computation of the inverse Hessian matrix at each update step. Spall \cite{spall-jacobian} presented enhancements to the four-simulation
Hessian estimator of \cite{spall2000adaptive} using certain weighting and feedback mechanisms. These enhancements are seen to improve the performance
of the resulting scheme.

The class of SF algorithms was extended by Bhatnagar \cite{bhatnagar2007} to include two Newton-based algorithms governed by
standard Gaussian perturbations. As with the gradient estimator, the Hessian estimator was obtained from the idea that if one
convolves the Hessian with a multivariate Gaussian density, then from an integration-by-parts argument applied twice,
the same can be viewed as a convolution of the objective function with a scaled multivariate Gaussian. This results in a single-measurement
Hessian estimator - the same measurement also estimates the gradient. A two-measurement SF algorithm presented there,
involving a balanced (two-measurement) Hessian estimator, is seen to work better in practice - again the same two measurements also estimate 
the gradient. In \cite{ghosh14a, ghosh14b}, gradient and Newton SF algorithms
based on the multi-variate $q$-Gaussian density as the smoothing functional, i.e., the perturbation distribution,  have been presented. This gives rise
to a class of smoothing densities parameterized by the $q$-parameter. Densities such as multivariate Normal, Cauchy and Uniform
that were known to satisfy the properties required of smoothing functionals \cite{rubinstein1993discrete} in SF algorithms, emerge as special cases
of the $q$-Gaussian density for different values of the parameter $q$. Thus, these papers have served to significantly extend the class
of perturbations that play the role of smoothing densities in SF algorithms.

Finally, the RDSA procedure has recently been revisited in detail by Prashanth et al. \cite{prashanth2017rdsa}, and novel gradient and Newton
algorithms have been devised. Recall that in the original RDSA procedure described in \cite{kushcla}, the perturbation variates are
required to be uniformly distributed over the surface of the unit sphere in $\R^N$, $N$ being the parameter dimension.
The approach taken in \cite{prashanth2017rdsa} involves a uniform distribution over a unit cube as opposed to the surface of the unit sphere.
The perturbation (component) random variables are thus allowed to be independent, symmetric, and uniformly distributed over an interval
that is symmetric around zero. Another class of perturbations, namely asymmetric Bernoulli, have been investigated and found to work
nearly as well as SPSA in both theory and practice. Hessian estimators derived from these perturbations have also been proposed in \cite{prashanth2017rdsa},
and both gradient and Newton algorithms have been investigated in detail.
The reader is referred to \cite{bhatnagar-book} for a rigorous introduction to the class of simultaneous perturbation methods.


In this paper, we are concerned with developing deterministic perturbation variants of first and second-order RDSA algorithms, henceforth referred to as RDSA-DP family of algorithms, with 1RDSA-DP (resp. 2RDSA-DP) denoting first (resp. second) order variants. 
The principal aim is to incorporate deterministic perturbation sequences into RDSA, such that the resulting gradient estimates are still asymptotically unbiased and the overall stochastic gradient algorithm converges, preferably at the same rate as that of the random perturbation RDSA counterparts. 
We consider two novel choices for deterministic perturbations - a semi-lexicographic sequence and a permutation matrix-based sequence. We combine the two sequences with first and second-order RDSA. 

In the case of 1RDSA-DP, the resulting algorithms, under both choices for deterministic perturbations, possess theoretical guarantees that are comparable to those of their random perturbation counterparts. This statement is true when we consider the asymptotic unbiasedness of the gradient estimation and asymptotic convergence of the overall 1RDSA-DP family of algorithms. Moreover, from a non-asymptotic bound that we derive for the special case of strongly-convex objective functions, we observe that the permutation matrix-based perturbations perform best, and even match the rate of a first-order method, whose gradients are directly available.

In the case of second-order RDSA, we incorporate both perturbation sequences to arrive at two variants of 2RDSA-DP, say 2RDSA-Lex-DP and 2RDSA-Perm-DP. However, the theoretical guarantees for the two variants differ significantly. For 2RDSA-Lex-DP, the asymptotic unbiasedness claim holds for the full Hessian, while a similar claim holds only for the Jacobi variant of 2RDSA-Perm-DP involving a diagonal matrix with diagonal elements being those of the Hessian. Furthermore, for the special case of a quadratic optimization problem in the noise-free regime, 2RDSA-Lex-DP is shown to exhibit a convergence rate that is comparable to that of 2SPSA with an adaptive feedback sequence that was proposed in \cite{spall-jacobian}. Note that a similar rate result does not exist for regular 2RDSA, and we believe, cannot be established.
In any case, to the best of our knowledge, no deterministic perturbation sequences exist for the class of second-order simultaneous perturbation algorithms, including the popular 2SPSA \cite{spall1997one}. 

In comparison to \cite{bhatfumarcwang}, which is the closest related work, we remark that (i) we propose a novel deterministic perturbation scheme and combine it with first-order and second-order RDSA, while the deterministic perturbation schemes in \cite{bhatfumarcwang} are only for first-order SPSA; (ii) unlike \cite{bhatfumarcwang}, we provide asymptotic normality results that quantify the convergence rate; and (iii) the permutation matrix-based perturbations that we propose are much easier to implement and require much less computational memory in comparison to the deterministic perturbation sequences proposed in \cite{bhatfumarcwang}. In particular, the permutation matrices have a linear dependence on the dimension $N$, while the lexicographic/Hadamard matrix-based perturbations in \cite{bhatfumarcwang} scale exponentially with $N$.

The rest of the paper is organized as follows: Section \ref{sec:1rdsa-dp} presents the first-order RDSA variants with two deterministic perturbation sequences, and Section \ref{sec:2rdsa-dp} describes deterministic perturbation variants of the second-order RDSA algorithm. The main theoretical guarantees for 1RDSA-DP and 2RDSA-DP algorithms are presented in Sections \ref{sec:1rdsa-dp}--\ref{sec:2rdsa-dp}, while Section \ref{sec:proofs} provides detailed convergence proofs. Section \ref{sec:expts} presents simulation experiments that compare the performance of the DP variants of RDSA with several algorithms that employ the simultaneous perturbation technique. Finally, Section \ref{sec:concl} presents the concluding remarks.

\section{First-order RDSA with deterministic perturbations (1RDSA-DP)}
\label{sec:1rdsa-dp}
A first-order method, given the gradient $\nabla f(\cdot)$, would feature an incremental update as follows:
\begin{align}
	x_{n+1} = x_n - a_n \nabla f(x_n).
	\label{eq:grad-des}
\end{align}
In the simulation optimization setting, we are given noisy function measurements, from which the gradient has to be estimated. The simultaneous perturbation method \cite{bhatnagar-book} is a popular approach for obtaining such gradients. The recently proposed RDSA algorithm estimates  $\nabla f(x_n)$ as follows:
\begin{align}
	\widehat\nabla f(x_n) = \frac1{1+\epsilon} d_n \left[ \dfrac{y_n^+ - y_n^-}{2\delta_n}\right],\label{eq:1rdsa-asymber-grad}
\end{align}
where $y_n^{\pm} = f(x_n\pm \delta_n d_n) + \xi$, $ \xi $ is the measure noise and $\delta_n$ is a perturbation constant. 
Further, $d_n = (d_n^1,\ldots,d_n^N)\tr$ is the random perturbation vector, with $d_n^i$, $i=1,\ldots,N$ chosen using the asymmetric Bernoulli distribution, i.e., 
$ d_n^i =   -1$ with probability (w.p.) $\dfrac{(1+\epsilon)}{(2+\epsilon)}$ and $1+\epsilon$ w.p. $\dfrac{1}{(2+\epsilon)}$ for some 
$\epsilon>0$.

In this paper, we propose a variant of 1RDSA 
that loops through a deterministic sequence to cancel out the bias in the gradient estimate -- a property that regular RDSA achieves in expectation through a zero-mean random perturbation. 
We consider two deterministic  constructions for the perturbations $d_n$. The first choice is based on a semi-lexicographic sequence, while the second employs permutation matrices. 
In both cases, we perform gradient descent similar to \eqref{eq:grad-des}, with a gradient estimate inspired from that of 1RDSA.  However, unlike \eqref{eq:1rdsa-asymber-grad} that has a random source for perturbations $d_n$, we loop through a deterministic sequence (cf. Tables \ref{fig:det-perturb} and \ref{fig:det-perturb-perm} below). 
\begin{table*}[t]
	\caption{Illustration of the deterministic perturbation sequence construction for two-dimensional and three-dimensional settings.}
	\label{fig:det-perturb}
		\begin{tabular}{c}
		\begin{subfigure}{0.2\textwidth}
			\centering
			\caption{Case $N=2$}
			\label{tab:2d}
				\scalebox{1.0}{
\begin{tabular}{|c|c|c|}
\toprule
Inner loop   counter $m$&$D_2^1$ & $D_2^2$\\
\midrule
$0$ &$-1$ & $-1$\\
$1$ &$-1$ & $-1$\\
$2$ &$-1$ & $2$\\
$3$ &$-1$ & $-1$\\
$4$ &$-1$ & $-1$\\
$5$ &$-1$ & $2$\\
$6$ &$2$ & $-1$\\
$7$ &$2$ & $-1$\\
$8$ &$2$ & $2$\\
\bottomrule
\end{tabular}}
\vspace{1ex}
\end{subfigure}
\\
		\begin{subfigure}{0.8\textwidth}
			\centering
			\caption{Case $N=3$}
			\label{tab:3d}
				\scalebox{1.0}{
\begin{tabular}{|c|c|c|c|c|c|c|c|c|c|c|c|}
\toprule
Inner loop  &$D_3^1$ & $D_3^2$ & $D_3^3$ & Inner loop    & $D_3^1$ & $D_3^2$ & $D_3^3$ & Inner loop    &$D_3^1$ & $D_3^2$ & $D_3^3$\\
counter $m$&        &         &         &  counter $m$  &         &         &         & counter $m$  &        &          &\\
\midrule
$0$   & $-1$ & $-1$ & $-1$ & $9$   & $-1$ & $-1$ & $-1$ &   $18$ & $2$ & $-1$ & $-1$\\  
$1$ & $-1$ & $-1$ & $-1$   &  $10$ & $-1$ & $-1$ & $-1$&   $19$ & $2$ & $-1$ & $-1$\\
$2$ & $-1$ & $-1$ & $2$    &   $11$ & $-1$ & $-1$ & $2$&    $20$ & $2$ & $-1$ & $2$\\
$3$ & $-1$ & $-1$ & $-1$   &  $12$ & $-1$ & $-1$ & $-1$&   $21$ & $2$ & $-1$ & $-1$\\
$4$ & $-1$ & $-1$ & $-1$   &  $13$ & $-1$ & $-1$ & $-1$&   $22$ & $2$ & $-1$ & $-1$\\
$5$ & $-1$ & $-1$ & $2$    &   $14$ & $-1$ & $-1$ & $2$&    $23$ & $2$ & $-1$ & $2$\\
$6$ & $-1$  & $2$ & $-1$   &  $15$ & $-1$  & $2$ & $-1$&   $24$ & $2$  & $2$ & $-1$\\
$7$ & $-1$  & $2$ & $-1$   &  $16$ & $-1$  & $2$ & $-1$&   $25$ & $2$  & $2$ & $-1$\\
$8$ & $-1$  & $2$ & $2$    &   $17$ & $-1$  & $2$ & $2$&    $26$ & $2$  & $2$ & $2$\\
\bottomrule
\end{tabular}}
\end{subfigure}
\end{tabular}
\end{table*}

\subsection{Semi-lexicographic sequence-based perturbations}

Algorithm \ref{alg:1rdsadp} presents the pseudocode for the 1RDSA-Lex-DP algorithm that employs a semi-lexicographic sequence for perturbations. 
Our proposed construction for perturbations $d_m$ is illustrated for the case when $N=2$ and $N=3$ in Tables \ref{tab:2d} and \ref{tab:3d}, respectively. 
Letting $\I_N$ denote the $N\times N$ identity matrix, for $N=2$, we have 
$$\sum\limits_{m=0}^{3^2-1} d_m d_m\tr = 
\left[\begin{array}{cc}
18 & 0\\
0 & 18 \\
\end{array}\right] \Longrightarrow \frac1{2\times3^2}\sum\limits_{m=0}^{3^2-1} d_m d_m\tr = \I_{2}. $$
In a similar fashion, for $N=3$, we have 
$$\sum\limits_{m=0}^{3^3-1} d_m d_m\tr = 
\left[\begin{array}{ccc}
54 & 0 & 0\\
0 & 54 & 0 \\
0 & 0 & 54  \\
\end{array}\right] \Longrightarrow \frac1{2\times3^3}\sum\limits_{m=0}^{3^3-1} d_m d_m\tr = \I_{3}.$$

For any $N$, we require that $\frac1{2\times3^N}\sum\limits_{m=0}^{3^N-1} d_m d_m\tr = \I_{N}$
to ensure that the gradient estimate $\widehat\nabla f(x_n)$ (see \eqref{eq:1rdsa-grad} in Algorithm \ref{alg:1rdsadp} is asymptotically unbiased.
The crucial ingredient in the asymptotic-unbiasedness proof, presented later in Lemma \ref{lemma:1rdsa-dp-bias}, is the following step that uses suitable Taylor's series expansions:
\begin{align*}
	d_m\left[\dfrac{f(x_n+\delta_n d_m) - f(x_n-\delta_n d_m)}{2\delta_n}\right] 
	=d_m d_m\tr \nabla f(x_n)\!+\! O(\delta_n^2).
\end{align*}
Hence, if the product $d_m d_m\tr$ sums to identity over a loop, then $\widehat\nabla f(x_n)$ would be asymptotically unbiased.

We now present the deterministic perturbation sequence for a general $N$.
Set $D^{1}_{1} = \left[ 
\begin{array}{c} 
-1 \\ -1\\2
\end{array}
\right]$ and apply the following recursion $N-1$ times to obtain $ D^{i}_{N}$:
\begin{align}
	D^{i}_{k+1} = \left[ 
	\begin{array}{c} 
		D^{i-1}_{k} \\ D^{i-1}_{k}\\D^{i-1}_{k}
	\end{array}
	\right], \quad i=2,\ldots,k+1. \label{eq:somelabel01}
\end{align}
The deterministic perturbation sequence loops through the rows in the matrix, say $D_N$, with columns $D^{i}_{N}$. 
Notice that each column $D^{i}_{N}, i=1,\ldots,N$ in $D_N$ is of length $3^{N}$. Further, in the first column of $D_N$, the first $2\times 3^k$ elements are $-1$ and the remaining $3^k$ elements are $2$. On the other hand, the columns $2$ through $N$ in $D_N$ are obtained from $D^{1}_{N-1},\ldots,D^{N-1}_{N-1}$, respectively by concatenating the $D^{i}_{N-1}$ columns thrice.

\begin{algorithm}
	\begin{algorithmic}
		\State {\bf Input:}  initial parameter $x_0 \in \R^N$, perturbation constants $\delta_n>0$,  step-sizes $a_n$, deterministic perturbations $\{d_0,\ldots,d_{3^N-1}\}$.
		\For{$n = 0,1,2,\ldots$}
		\State{}
		\Comment{\textit{Fix $x_n$ and loop through the rows of matrix $ D_N $ for perturbations $d_m$.}}
		\For{$m = 0,1,2,\ldots,3^N-1$}	
	\State Obtain function values $y_m^+ = f(x_n + \delta_{n 3^{N} + m} d_m) + \xi$ and $y_m^- = f(x_n - \delta_{n 3^{N} + m} d_m) + \xi$, where $\xi$ is the measure noise.
	\State Set $g_m = d_m \left[ \dfrac{y_m^+ - y_m^-}{2\delta_{n 3^{N} + m}}\right].$
	\EndFor
		
		\vspace{-4ex}
		
		\begin{align}
			\label{eq:1rdsa-grad}
			\text{Gradient estimate:}& \qquad		\widehat\nabla f(x_n) = \frac1{2\times3^N}  \sum\limits_{m=0}^{3^N-1} g_m.\\
			\text{Parameter update:}&\qquad
			\label{eq:1rdsa}
			x_{n+1} = x_n - a_n \widehat\nabla f(x_n).
		\end{align}
		\EndFor
		\State {\bf Return} $x_n$.
	\end{algorithmic}
\caption{1RDSA-Lex-DP}
	\label{alg:1rdsadp}
\end{algorithm}

\subsection{Permutation matrix-based perturbations}
While the semi-lexicographic sequence-based perturbations result in a gradient estimate that is asymptotically unbiased, the inner loop for the perturbations becomes exponentially longer as a function of the dimension $N$. This exponential dependence on $N$ is problematic, because the descent in parameter $x_n$ occurs at the end of the inner loop (see Algorithm \ref{alg:1rdsadp}), and hence, a long inner loop would imply slow updates (and slow convergence) to $x_n$. 

In this section, we propose 
an efficient alternative to the semi-lexicographic deterministic sequence; the approach is based on permutation matrices.
A permutation matrix is a matrix whose rows are the rows of an identity matrix in some order.
For instance, the permutation matrices in two dimension are
\begin{align*}
	\left[\begin{array}{ccc}
		1 & 0 \\
		0 & 1 \\
	\end{array}\right]
	& \textrm{ and } 
	\left[\begin{array}{ccc}
		0 & 1 \\
		1 & 0 \\
	\end{array}\right].
\end{align*}
In three dimensions, there are $6$ permutation matrices. In general, there are $N!$ permutation matrices in dimension $N$.

In the case of permutation matrix-based deterministic perturbations, the overall algorithm follows the template provided in Algorithm \ref{alg:1rdsadp}, except that the perturbations are generated using a permutation matrix in $N$-dimensions, the inner loop for $m$ runs from $0$ to $N-1$ and the gradient estimate in \eqref{eq:1rdsa-grad} is replaced by
\begin{align}
	\label{eq:1rdsa-grad-perm}
	\widehat\nabla f(x_n) =  \sum\limits_{m=0}^{N-1} g_m.
\end{align}

Table \ref{fig:det-perturb-perm} illustrates the perturbations $d_m$ used in Algorithm \ref{alg:1rdsadp}, for $N=2$ and $N=3$. In a nutshell, the sequence shown in Table \ref{fig:det-perturb-perm} loops through the rows of the identity matrix in some order.
\begin{table}[h]
	\caption{Illustration of the permutation matrix-based deterministic perturbation sequence construction for two-dimensional and three-dimensional settings.}
	\label{fig:det-perturb-perm}
    \centering
		\begin{tabular}{ccc}
		\begin{subfigure}{0.3\textwidth}
			\centering
			\caption{Case $N=2$}
			\label{tab:2dperm}
\begin{tabular}{|c|c|c|}
\toprule
Inner loop  &$D_2^1$ & $D_2^2$\\
 counter $m$  && \\
\midrule
$0$ &$1$ & $0$\\
$1$ &$0$ & $1$\\
\bottomrule
\end{tabular}
\end{subfigure}
& &
\begin{subfigure}{0.3\textwidth}
\centering
\caption{Case $N=3$}
\label{tab:3dperm}
\begin{tabular}{|c|c|c|c|c|c|c|c|c|c|c|c|}
\toprule
Inner loop  &$D_3^1$ & $D_3^2$ & $D_3^3$ \\
counter $m$&        &         &        \\
\midrule
$0$   & $0$ & $1$ & $0$\\  
$1$ & $0$ & $0$ & $1$ \\
$2$ & $1$ & $0$ & $0$ \\
\bottomrule
\end{tabular}
\end{subfigure}
\end{tabular}
\end{table}

\begin{remark}
The classic Kiefer-Wolfowitz (K-W) algorithm \cite{kw} obtains $2N$ function samples per iteration, corresponding to parameters $x_n \pm \delta_n e_i$, $i=1,\ldots,N$ and updates the parameter as follows:
\[ x^i_{n+1} = x^i_n  - a_n \left( \dfrac{y_n^{i+} - y_n^{i-}}{2\delta_n}\right),\]
where $y_n^{i\pm} = f(x_n \pm \delta_n e_i)$, $i=1,\ldots, N$. 

The 1RDSA-Perm-DP algorithm that we propose resembles K-W in the sense that the inner loop obtains $2N$ samples before updating the parameter $x_n$. However, the gradient estimate features a product with the perturbation vector $d_m$ and this is unlike K-W, where the individual coordinates are independently updated.  
\end{remark}
%

\subsection{Main results}
Let $D^{1}_{N},\ldots, D^{N}_{N}$ denote the $N$ columns of the semi-lexicographic perturbation variables. Consider the matrix
\[
M_N = \left[ 
\begin{array}{cccc}
(D^{1}_{N})\tr D^{1}_{N} & (D^{1}_{N})\tr D^{2}_{N} & \ldots & (D^{1}_{N})\tr D^{N}_{N}\\
(D^{2}_{N})\tr D^{1}_{N} & (D^{2}_{N})\tr D^{2}_{N} & \ldots & (D^{2}_{N})\tr D^{N}_{N}\\
\vdots & \vdots & & \vdots\\
(D^{N}_{N})\tr D^{1}_{N} & (D^{N}_{N})\tr D^{2}_{N} & \ldots & (D^{N}_{N})\tr D^{N}_{N}\\
\end{array}
\right].
\]

\begin{lemma}
	\label{lemma:MNid}
For 1RDSA-Lex-DP $M_N = 2 \times 3^N \I_N$, and for 1RDSA-Perm-DP $M_N=\I_N$. 
\end{lemma}
\begin{proof}
	See Section \ref{sec:1rdsa-dp-proofs-MNid}.
\end{proof}

Before providing the convergence claims for 1RDSA-DP with either perturbation choice, we outline the necessary assumptions below.
\begin{enumerate}[label=(\textbf{A\arabic*})]
\item $f:\R^N\rightarrow \R$ is three-times continuously differentiable\footnote{Here $\nabla^3 f(x) = \dfrac{\partial^3 f (x)}{\partial x\tr \partial x\tr \partial x\tr}$ denotes the third derivative of $f$ at $x$ and $\nabla^3_{i_1 i_2 i_3} f(x)$ denotes the $(i_1 i_2 i_3)$th entry of $\nabla^3 f(x)$, for $i_1, i_2, i_3=1,\ldots, N$.}  with $\left|\nabla^3_{i_1 i_2 i_3} f(x_n) \right| < \alpha_0 < \infty$, for $i_1, i_2, i_3=1,\ldots, N$ and for all $ n $. 
	\item $\{\xi_m^+,\xi_m^-, m=0,\ldots,P, n=1,2,\ldots\}$ satisfy $\E\left[\left.\xi_m^+ - \xi_m^- \right| \F_n\right] = 0$, where $P=3^{N}-1$ for semi-lexicographic 1RDSA-DP and $P=N-1$ for permutation matrix-based 1RDSA-DP. 
\item For some $\alpha_0, \alpha_1 >0$ and for all $m$, $ n $, 
$\E \left|\xi_m^{\pm}\right|^{2} \le \alpha_0$, $\E \left|f(x_n\pm \delta d_m)\right|^{2} \le \alpha_1$ for any $\delta >0$ and $d_m$, $m=0,\ldots,P$. 
	\item The step-sizes $a_n$ and perturbation constants $\delta_n$ are positive, for all $n$ and satisfy
	\[\hspace{-1.1em}a_n, \delta_n \rightarrow 0\text{ as } n \rightarrow \infty, 
	\sum_n a_n=\infty \text{ and } \sum_n \left(\frac{a_n}{\delta_n}\right)^2 <\infty.\]
	\item $\sup_n \left\| x_n \right\| < \infty$ w.p. $1$.
\end{enumerate}
The assumptions above are common to the analysis of simultaneous perturbation methods, and can be found, for instance, in the context of 1SPSA \cite{spall1992multivariate} -- see also \cite{bhatnagar-book} for the analysis of other simultaneous perturbation schemes that employ similar assumptions. 
The first two are necessary to establish asymptotic unbiasedness of the 1RDSA-DP gradient estimate through a Taylor series expansion facilitated by (A1), while ignoring the noise owing to (A2). The third and fourth assumptions are necessary to ignore the effects of noise on the convergence behavior of $x_n$. The final assumption requiring boundedness of the iterates $x_n$ can be ensured by constraining the iterates $ x_n $ to evolve in a certain compact region and projecting them back each time they go out of the region, see \cite{kushcla} (Chapter 5).
If the projected region contains the optima, then the stochastic gradient algorithms (RDSA or SPSA) would converge to this point, and in the complementary case, the algorithm would get stuck on the boundary of the projection region. In the literature, there also exist approaches to overcome the latter case, by either growing the projection region \cite{chen1987convergence}, or performing sparse projections \cite{dalal2017finite} (i.e., at time instants that grow exponentially to infinity, while not projecting the iterates at the remaining time instants). 

\begin{lemma}(\textbf{\textit{Asymptotic unbiasedness of 1RDSA-DP gradient estimate}})
	\label{lemma:1rdsa-dp-bias}
	Under (A1)-(A5), 
	\begin{enumerate}[label=(\textbf{\roman*})]
		\item for $\widehat\nabla f(x_n)$ defined according to \eqref{eq:1rdsa-grad}, we have a.s. that\footnote{Here $\widehat\nabla_i f(x_n)$ and $\nabla_i f(x_n)$ denote the $i$th coordinates in the gradient estimate $\widehat\nabla f(x_n)$ and true gradient $\nabla f(x_n)$, respectively.}
		\begin{align*}
			\left| \E\left[\left.\widehat\nabla_i f(x_n)\right| \F_n \right] - \nabla_i f(x_n)\right| = C_0 \delta_{n3^N}^2,   
		\end{align*} 
		for  $i=1,\ldots,N$, where $ C_0 = \alpha_0 N^3 3^{N-1} $ and $\F_n = \sigma(x_k,k\le n)$, $n\ge 1$.
		
		\item for $\widehat\nabla f(x_n)$ defined according to \eqref{eq:1rdsa-grad-perm}, we have a.s. that
		\begin{align}
			\left| \E\left[\left.\widehat\nabla_i f(x_n)\right| \F_n \right] - \nabla_i f(x_n)\right| = C_0 \delta_{nN}^2,
		\end{align} 
		for $i=1,\ldots,N$, where $ C_0 = \alpha_0 N^3/6 $.
	\end{enumerate}
\end{lemma}
\begin{proof}
	See Section \ref{sec:1rdsa-dp-proofs-bias}.
\end{proof}

The advantage of the permutation matrix approach is that the dependence on the dimension $N$ is linear, whereas the semi-lexicographic sequence has an exponential dependence on $N$. 

We now have an asymptotic convergence claim for $x_n$ updated according to \eqref{eq:1rdsa}; the claim is verbatim from Theorem 2 of \cite{prashanth2017rdsa}. 
\begin{theorem}(\textbf{Strong Convergence})
	\label{thm:1rdsa-strong-conv}
	Let $x^*$ be an  asymptotically stable equilibrium of the following ordinary differential equation (ODE):
	$
	\dot{x}_t = -\nabla f(x_t),$ with  domain of attraction $D(x^*)$, i.e., $D(x^*)= \{x_0 \mid \lim_{t\rightarrow\infty} x(t\mid x_0) = x^*\}$, where $x(t\mid x_0)$ is the solution to the ODE with initial condition $x_0$. Assume (A1)-(A5), and also that there exists a compact subset $\mathcal D$ of $D(x^*)$ such that $x_n \in \mathcal D$ infinitely often. Let $x_n$ be governed by \eqref{eq:1rdsa}, with the gradient estimate $\widehat\nabla f(x_n)$ defined either according to \eqref{eq:1rdsa-grad} or \eqref{eq:1rdsa-grad-perm}. Then,  
	\[x_n \rightarrow x^* \text{ a.s. as } n\rightarrow \infty.\] 
\end{theorem}
\begin{proof}
	See Section \ref{sec:1rdsa-dp-proofs-strong-conv}.
\end{proof}

%

For the special case when the objective $f$ is strongly-convex, we present a non-asymptotic bound for 1RDSA-DP with permutation matrix-based perturbations. 
More precisely, we assume the objective function $f$ satisfies the following assumption:

\noindent (\textbf{A1'}) For any $x, x'$, we have 
\[(\nabla f(x) - \nabla f(x'))\tr (x - x') \ge \mu \l x - x' \r^2,\] for some $\mu >0$.

\begin{theorem}
\label{thm:exp-bound}
Under (A1') and (A2)-(A5),  we have,
 \begin{align}
\E  \l x_{n+1} - x^* \r
&  	\le \underbrace{\sqrt{2}\e(-\mu \Gamma_n) \l x_0 - x^* \r}_{\textbf{initial error}}\nonumber\\
		  &+  \left(\underbrace{3\sum\limits_{k=1}^{n}a^2_{k}\e(-2\mu(\Gamma_n - \Gamma_{k})) C_0^2 \delta_k^4}_{\textbf{bias error}}\right.
		  \left.+ \underbrace{2\sum\limits_{k=1}^{n}a^2_{k}
		  							\e(-2\mu(\Gamma_n - \Gamma_{k})) C_1 \delta_k^{-2}}_{\textbf{sampling error}} \right)^{\frac{1}{2}}
		  							,\label{eq:expectation-bound}
 \end{align} 
where $x^*$ is the global minimizer of $f$, $\Gamma_k:=\sum_{i=1}^k a_i$, $ C_0 $ is as defined in Lemma \ref{lemma:1rdsa-dp-bias}, and $ C_1 = \alpha_1 N/2$.
\end{theorem}
\begin{proof}
 See Section \ref{sec:1rdsa-dp-proofs-exp-bound}.
\end{proof}
The initial error depends on the starting point $x_0$ of the algorithm. The sampling error relates to a martingale difference sequence, which arises due to the fact that only noisy measurements of the objective function are available. The bias error arises out of the need to estimate gradients from function measurements and quantifies the error in gradient estimation. The initial and sampling error components are common to classic stochastic convex optimization settings, while the bias error is specific to the simulation optimization framework, i.e., a setting where gradients are not directly available and have to be estimated from noisy function measurements.

Now we specialize the result above by choosing the step-size $a_n$ and perturbation constant $\delta_n$ to obtain an order $O\left(\frac{1}{\sqrt{n}}\right)$ bound in expectation on the optimization error of the algorithm.
\begin{theorem}
\label{thm:1rdsa-dp-rate}
Let $a_k = c/k$ and $\delta_k = \delta_0/k^{\delta}$. Then under (A1'), (A2) and (A3),
\begin{align*}
\E \l x_n - x^* \r &\le  \dfrac{\sqrt{2}\l x_0 - x^*\r}{n^{\mu c}} + \dfrac{\sqrt{3} c C_0  \delta_0^2}{\sqrt{2\mu c - 4\delta-1}} n^{-\frac{(1+4\delta)}{2}} +  \dfrac{\sqrt{3 C_1} c }{\delta_0 \sqrt{2\mu c + 2 \delta -1}} n^{\delta-\frac{1}{2}}.
\end{align*}
\end{theorem}
\begin{proof}
 See Section \ref{sec:1rdsa-dp-proofs-dp-rate}.
\end{proof}
Choosing $\delta=0$, one can recover the optimal rate of the order $O\left(n^{-1/2}\right)$ for simultaneous perturbation schemes. 
Further, choosing $c$ such that $\mu c > 1/2$, it is easy to observe that the initial error is forgotten faster than the other error components. 
In contrast, for the more general case of non-convex objective $f$, the authors in \cite{spall1992multivariate,chin1997comparative} are able to establish a rate of  $O\left(n^{-1/3}\right)$ obtained from an asymptotic mean square error analysis using the second moment of the limiting normal distribution. More recently, for the case of convex (and not necessarily strongly-convex) objective $f$, an error of the order $O\left(n^{-1/3}\right)$ is unavoidable from an information-theoretic (or minimax) viewpoint -- see \cite{HuPrGySz16} for further details.

\section{Second-order RDSA with deterministic perturbations (2RDSA-DP)}
\label{sec:2rdsa-dp}
Second-order methods provide many advantages over their first-order counterparts. The main benefit of second-order methods over first-order methods is that they converge at the optimum rate without requiring knowledge of minimum eigenvalue of  $ \nabla^2 f(x^*)$. Other benefits include (i) faster convergence in the final phase, i.e., when the iterate is close to the optima as second-order methods minimize a quadratic model of $f$ and (ii) scale-invariance, i.e., second-order methods adjust automatically to the scale of the parameter and hence, the update rule is unaffected. On the flip side, second-order schemes require estimating the Hessian in addition to the gradient of $f$ and have a higher per-iteration cost due to matrix inversion. 

Using the two deterministic sequences, i.e., semi-lexicographic and permutation matrix-based choices, presented in the previous section, we provide two variants of the 2RDSA algorithm proposed in \cite{prashanth2017rdsa}. 

\subsection{Semi-lexicographic sequence-based perturbations}

The reason such deterministic choices for perturbations work in the context of 2RDSA can be seen as follows: 
Using suitable Taylor's series expansions (see Lemma \ref{lemma:2rdsa-dp-bias} below), we have
\begin{align}
\E&[\widehat H_n \mid \F_n] =   \frac1{2\times3^N}  \sum\limits_{m=0}^{3^N-1} \left( M_m \sum\limits_{i=1}^{N} (d_m^i)^2 \nabla^2_{ii} f(x_n) \right.
\left.+ 2\sum\limits_{i=1}^{N-1}\sum\limits_{j=i+1}^N d_m^i d_m^j \nabla^2_{ij} f(x_n) + O(\delta_n^2)\right). \label{eq:h1m}
\end{align}

From Lemma \ref{lemma:MNid}, it can be seen that
\[\frac1{2\times3^N}\sum\limits_{m=0}^{3^N-1} \sum\limits_{i=1}^{N} (d_m^i)^2 = 1,\text{ and }
\sum\limits_{m=0}^{3^N-1} \sum\limits_{i=1}^{N-1}\sum\limits_{j=i+1}^N d_m^i d_m^j = 0.\]
The above equality can be directly verified for the case when $N=2$ and $N=3$ using Table \ref{fig:det-perturb}. 

Plugging the fact above into \eqref{eq:h1m} followed by a tedious calculation (see Lemma \ref{lemma:2rdsa-dp-bias} below), we obtain
\begin{align*}
\E[\widehat H_n(i,j) \mid \F_n] = \nabla^2_{ij} f(x_n) + O(\delta_{n3^N}^2).
\end{align*}

\begin{remark}\textit{\textbf{(Jacobi variant)}}
	If the Hessian is known to be in a diagonal form, i.e., if the requirement is for an algorithm to estimate $\nabla^2_{ii} f(\cdot)$, then the estimate of Algorithm \ref{alg:2rdsa-dp} can be replaced by the following:
	\[\widehat H_n = \frac1{2\times3^N}  \sum\limits_{m=0}^{3^N-1} \tilde H_m,\]
	with the inner-loop Hessian estimate given by
	\[ \tilde H_m = \left[ \dfrac{y_m^+ + y_m^- - 2 y_n}{\delta_{n3^N + m}^2}\right]. \]
	Notice that, unlike Algorithm \ref{alg:2rdsa-dp}, the scheme above (the so-called Jacobi variant of stochastic Newton algorithms - cf. \cite{bhatnagar2005}) can estimate the diagonal entries of the Hessian, and, more importantly, cannot estimate the off-diagonal entries of the Hessian, as the off-diagonal perturbation terms of interest zero out over the inner loop, i.e., $\sum\limits_{m=0}^{3^N-1} \sum\limits_{i=1}^{N-1}\sum\limits_{j=i+1}^N d_m^i d_m^j =0$. 
\end{remark}

	\begin{algorithm}
		\begin{algorithmic}
			\State {\bf Input:}  initial parameter $x_0 \in \R^N$, perturbation constants $\delta_n>0$,   step-sizes $\{a_n, b_n\}$, matrix projection operator $ \Upsilon $.
			The deterministic perturbation $\{d_m\}$ sequence is chosen in the same manner as in 1RDSA-DP.
			\For{$n = 0,1,2,\ldots$}
			\State Obtain function value $y_n = f(x_n) + \xi$, where $\xi$ is the measure noise.
			\State{}
			\Comment{As in 1RDSA, fix $x_n$ and loop through the rows of matrix $ D_N $ for perturbations $d_m$.}
			\For{$m = 0,1,2,\ldots,3^N-1$}	
			\State Obtain function values $y_m^+ = f(x_n + \delta_{n 3^{N} + m} d_m) + \xi$ and $y_m^- = f(x_n - \delta_{n 3^{N} + m} d_m) + \xi$, where $\xi$ is the measure noise.
			\vspace{-3ex}
			
			\State \begin{align}
				\indent	&\text{Set } \quad g_m = d_m \left[ \dfrac{y_m^+ - y_m^-}{2\delta_{n3^N + m}}\right],\\
				&\text{Set }\quad		
				\tilde H_m = M_m \left[ \dfrac{y_m^+ + y_m^- - 2 y_n}{\delta_{n3^N + m}^2}\right], \textrm{ where }\\[1ex]
				&M_m \!=\!
				\left[\!\!
				\begin{array}{ccc}
					\kappa\left((d_m^1)^2\!-\!2\times 3^{N}\right) & \cdots & d_m^1 d_m^N\\
					d_m^2 d_m^1  &  \cdots & d_m^2 d_m^N\\
					\vdots & \vdots & \vdots\\
					d_m^N d_m^1 & \cdots &  \kappa\left((d_m^N)^2\!-\!2\times 3^{N}\right) \\
				\end{array}
				\!\!\right],\nonumber\\
				&\textrm{and }\kappa = \left(\dfrac{1}{2\times 3^{N-1}} - 1\right)^{-1}.\nonumber
			\end{align}
			\EndFor
			\begin{align}
				\text{Gradient estimate:}& \quad		\widehat\nabla f(x_n) = \frac1{2\times3^N}  \sum\limits_{m=0}^{3^N-1} g_m.	\label{eq:1rdsa-grad-in-2rdsa} \\
				\text{Hessian estimate:}& \quad				\widehat H_n = \frac1{(2\times3^N)^2}  \sum\limits_{m=0}^{3^N-1} \tilde H_m.\label{eq:2rdsa-estimate-dp}\\
				\text{Hessian update:}& \quad		\overline H_n = \; (1-b_{n})  \overline H_{n-1} + b_{n} \widehat H_n. \label{eq:2rdsa-H}\\
				\text{Parameter update:}& \quad		x_{n+1} = \; x_n - a_n \Upsilon(\overline H_n)^{-1}\widehat\nabla f(x_n). \label{eq:2rdsa-x} \end{align}
			\EndFor
			\State {\bf Return} $x_n$.
		\end{algorithmic}
		\caption{2RDSA-Lex-DP}
		\label{alg:2rdsa-dp}
	\end{algorithm}

\subsection{Permutation matrix-based perturbations}
In the case of permutation matrix-based deterministic perturbations, it is not possible to estimate the off-diagonal entries of the Hessian. 
This is because $\sum\limits_{m=0}^{N-1} d_m^i d_m^j =0,\ i\ne j$ for permutation matrix-based perturbations. While a similar property holds for semi-lexicographic perturbations as well, we could add correction factors through the $M_m$ matrix (see Algorithm \ref{alg:2rdsa-dp}) to produce an estimate for all the entries in the Hessian matrix. A similar correction factor is not feasible for the case of permutation matrix-based perturbations, because each column of a permutation matrix contains only one positive ($=1$) entry, while the rest are zero. In other words, in \eqref{eq:h1m}, the second term inside the brackets always sums to zero when the outside summation for $m$ goes up to $N-1$, irrespective of the choice for $M_m$. 

However, using a permutation matrix, it is possible to estimate the diagonal entries of the Hessian. In this case, 
 the overall algorithm follows the template provided in Algorithm \ref{alg:2rdsa-dp}, except that the perturbations are generated using a permutation matrix in $N$-dimensions, the inner loop for $m$ would run from $0$ to $N-1$, and the gradient/Hessian estimates in Algorithm \ref{alg:2rdsa-dp} are replaced by
 		\begin{align}
		\label{eq:1rdsa-grad-perm-2}
		\text{Gradient estimate:}& \quad	\widehat\nabla f(x_n) =  \sum\limits_{m=0}^{N-1} g_m,\\
\text{Hessian estimate:}& \quad				\widehat H_n = \sum\limits_{m=0}^{N-1} \tilde H_m,\label{eq:2rdsa-estimate-perm-dp}
\end{align}
where 
\[ \tilde H_m = \left[ \dfrac{y_m^+ + y_m^- - 2 y_n}{\delta_{n N + m}^2}\right]. \]

\subsection*{Alternative using two permutation matrices}
Let $D_N$ and $\hat D_N$ be two $N$-dimensional permutation matrices that are not identical.  
Let $y_m^+ = f(x_n + \delta_n d_m + \delta_n \hat d_m) + \xi$ and $y_m^- = f(x_n - \delta_n d_m - \delta_n \hat d_m) + \xi$ be function measurements, where $ \xi $ is the measure noise, $d_m$ and $\hat d_m$ are sourced from $D_N$ and $\hat D_N$, respectively. In other words, $d_m$ and $\hat d_m$ would loop through the rows of $D_N$ and $\hat D_N$, respectively.
Consider the following estimate for the Hessian, in place of \eqref{eq:2rdsa-estimate-dp}:
\begin{align}
	\text{Hessian estimate:}& \quad				\widehat H_n = \sum\limits_{m=0}^{N-1} \tilde H_m,\label{eq:2rdsa-estimate-perm-dp-2}
\end{align}
where 
\[ \tilde H_m = \left[ \dfrac{y_m^+ + y_m^- - 2 y_n}{\delta_{n N + m}^2}\right]. \]

\subsection{Main results}

The analysis of 2RDSA-DP is under assumptions that match those employed for studying the convergence behavior of regular second-order SPSA and RDSA algorithms (i.e., with random perturbations), and we list them below for the sake of completeness.

\begin{enumerate}[label=(\textbf{C\arabic*})]
\item  The function
$f$ is four-times differentiable\footnote{Here $\nabla^4 f(x) = \dfrac{\partial^4 f (x)}{\partial x\tr \partial x\tr \partial x\tr \partial x\tr}$ denotes the fourth derivative of $f$ at $x$ and $\nabla^4_{i_1 i_2 i_3 i_4} f(x)$ denotes the $(i_1 i_2 i_3 i_4)$th entry of $\nabla^4 f(x)$, for $i_1, i_2, i_3,i_4=1,\ldots, N$.} with $\left|\nabla^4_{i_1 i_2 i_3 i_4} f(x_n) \right| < \infty$, for $i_1, i_2, i_3,i_4=1,\ldots, N$ and for all $n$. 

\item For each $n$ and all $x$, there exists a $\rho>0$ not dependent on $n$ and $x$, such that $(x-x^*)\tr \bar f_n(x) \ge \rho \left\| x_n - x\right\|$, where $\bar f_n(x) = \Upsilon(\overline H_n)^{-1} \nabla f(x)$.

\item $\{\xi_m, \xi_m^+,\xi_m^-, m = 0,\ldots,P, n=1,2,\ldots\}$ satisfy $\E\left[\left. \xi_m^+ + \xi_m^- - 2 \xi_m \right| \F_n\right] = 0$, where $P=3^{N}-1$ for semi-lexicographic 2RDSA-DP and $P=N-1$ for permutation matrix-based 2RDSA-DP. 

\item Same as (A4).

\item For each $i=1,\ldots,N$ and any $\rho>0$, 
$P(\{ \bar f_{ni} (x_n) \ge 0 \text{ i.o}\} \cap \{ \bar f_{ni} (x_n) < 0 \text{ i.o}\} \mid \{ |x_{ni} - x^*_i| \ge \rho\quad \forall n\}) =0.$

\item The operator $\Upsilon$ satisfies $\delta_n^2 \Upsilon(H_n)^{-1} \rightarrow 0$ a.s. and  $E(\left\| \Upsilon(H_n)^{-1}\right\|^{2+\zeta}) \le \rho$ for some $\zeta, \rho>0$.

\item For any $\varsigma >0$ and nonempty $S \subseteq \{1,\ldots,N\}$, there exists a $\rho'(\varsigma,S)>\varsigma$ such that 
$$ \limsup_{n\rightarrow \infty} \left| \dfrac{\sum_{i \notin S} (x-x^*)_i \bar f_{ni}(x)}{\sum_{i \in S} (x-x^*)_i \bar f_{ni}(x)}               \right| < 1 \text{ a.s.}$$
for all $|(x-x^*)_i| < \varsigma$ when $i \notin S$ and $|(x-x^*)_i| \ge \rho'(\varsigma,S)$ when $i\in S$.
\item For some $\alpha_0, \alpha_1>0$ and for all $m$, $ n $, $\E |{\xi_m}|^{2} \le \alpha_0$, $\E{|\xi_m^{\pm}|^{2}} \le \alpha_0$, $\E |f(x_n)|^{2}\le \alpha_1$ and $\E |f(x_n \pm \delta d_m)|^{2} \le \alpha_1$, for any $ \delta > 0 $. 
\item  $\sum_n \frac{1}{(n+1)^{2}\delta_n^{4}} < \infty$.
\end{enumerate}
\textit{Comments on assumptions (C1)-(C9)}:
(C1) and (C2) are basic assumptions about the smoothness and steepness of the function $ f $. (C1) holds if $ f $ is twice continuously differentiable with a bounded second derivative on $ \R^N $ and (C2) ensures the function $ f $ has enough curvature. 
(C3) and (C4) are common martingale-difference noise and step-sizes conditions and can be motivated in a similar manner as in the case of 1RDSA-DP (see Section II-C). 
(C5) says that if $ x_n $ is uniformly bounded away from  $ x^* $, then $ x_n $  cannot be bouncing around causing the change in signs of the normalized gradient elements an infinite number of times. 
(C6) can be ensured by having $\Upsilon (A)$ defined as performing an eigen-decomposition of $ A $ followed by projecting the eigenvalues to the positive side by adding a large enough scalar. 
(C7) ensures that, after sufficiently large iterations, each element of $ \bar{f}_n(x) $ tends to make a non-negligible contribution to products of the form $ (x-x^*)^T\bar{f}_n(x)  $ (see C2). 
(C5) and (C7) are not necessary if the iterates are bounded, i.e., $\sup _ { n } \left\| x _ { n } \right\| < \infty$ a.s. 
Finally, (C8) and (C9) are necessary to ensure convergence of the Hessian recursion, in particular, to invoke a martingale convergence result (see Theorem 7 and the proof of \cite[Theorem 2a]{spall2000adaptive}).
For a more detailed interpretation of the above conditions, the reader is referred to Section III and Appendix B of \cite{spall2000adaptive}. 

The main claim that establishes the asymptotic unbiasedness of the Hessian estimate in the DP variants of 2RDSA that we propose is given below.
\begin{lemma}(\textbf{\textit{Asymptotic unbiasedness of 2RDSA-DP Hessian estimate}})
\label{lemma:2rdsa-dp-bias}
Under (C1)-(C9), 
\begin{enumerate}[label=(\textbf{\roman*})]
\item for $\widehat H_n$ defined according to \eqref{eq:2rdsa-estimate-dp}, we have a.s. that\footnote{Here $\widehat H_n(i,j)$ and $\nabla^2_{ij}f(\cdot)$ denote the $(i,j)$th entry in the Hessian estimate $\widehat H_n$ and the true Hessian $\nabla^2 f(\cdot)$, respectively.}, for $i,j = 1,\ldots,N$,
\begin{align}
\left|\E\left[
\left. \widehat H_n(i,j) \right| \F_n \right] - \nabla^2_{ij} f(x_n)\right| = O(\delta_{n3^N}^2).
\end{align} 

\item for $\widehat H_n$ defined according to \eqref{eq:2rdsa-estimate-perm-dp}, we have a.s. that
\begin{align}
\left|\E\left[
\left. \widehat H_n(i,i) \right| \F_n \right] - \nabla^2_{ii} f(x_n)\right| = O(\delta_{nN}^2),
\end{align} 
for $i=1,\ldots,N$.
\end{enumerate}
\end{lemma}
\begin{proof}
See Section \ref{sec:2rdsa-dp-proofs-bias}.
\end{proof}

Once we have the asymptotic unbiasedness  for the 2RDSA-DP Hessian estimate, the convergence of the Hessian recursion is immediate and is given below for the sake of completeness.
\begin{theorem}(\textbf{Strong Convergence})
	\label{thm:2rdsa-dp-H}
Assume (C1)-(C9). Then $ x_n \rightarrow x^* $ a.s. as $n\rightarrow \infty$, where $ x_n $ is given by \eqref{eq:2rdsa-x}. For 2RDSA-Lex-DP $\overline H_n \rightarrow \nabla^2 f(x^*)$ a.s. as $n\rightarrow \infty$, furthermore, if the true Hessian is diagonal, then even for 2RDSA-Perm-DP $\overline H_n \rightarrow \nabla^2 f(x^*)$ a.s. as $n\rightarrow \infty$, where $\overline H_n$ is governed by \eqref{eq:2rdsa-H}. 
\end{theorem}
\begin{proof}
	See Section \ref{sec:2rdsa-dp-proofs-H}.
\end{proof}

We next present a convergence rate result for the special case of a quadratic objective function under the following additional assumptions, which have been used to establish a rate result for a variant of 2SPSA in \cite{spall-jacobian}:
\begin{enumerate}[label=(\textbf{C\arabic*}),resume]
\item  $f$ is quadratic and $\nabla^2 f(x^*) > 0$. 
\item The operator $\Upsilon$ is chosen such that $\E\parallel \Upsilon(\overline H_n) - \overline H_n\parallel^2 = o(e^{-2b_0n^{1-r}/(1-r)})$ and $\parallel \Upsilon(H) - H \parallel^2 / \left(1+\parallel H \parallel^2\right)$ is uniformly bounded.
\end{enumerate}
%
The assumptions (C10) and (C11) are much stronger compared to (C1) and (C6), respectively.
In a noise-free setting, after suitable Taylor series expansions, the Hessian estimate of 2SPSA can be written as 
\[ \widehat H_n(2SPSA) = \nabla^2 H(x_n) + \Psi(H(x_n)), \]
where $H(x_n)$ is the Hessian at iterate $x_n$ and $\Psi((H(x_n))$ is a function that involves the Hessian at $x_n$ and random perturbations used in 2SPSA. More importantly, $\Psi(H(x_n))$ is zero-mean. 
In \cite{spall-jacobian}, since the true Hessian $H(x_n)$ is not known in practice, in place of $\widehat H_n$ in the Hessian update in \eqref{eq:2rdsa-H}, the author uses the following improved Hessian estimate: $\widehat H_n(2SPSA) - \Psi(\overline H_n)$. The rationale underlying this replacement is that subtracting the $\Psi(\overline H_n)$ term reduces the error in Hessian estimation. 
In \cite{reddy2016improved}, the authors use a similar trick to reduce the Hessian estimation error in regular 2RDSA.
We claim such a feedback term is not necessary in the context of 2RDSA-DP and this can be argued as follows:
From the proof passage leading to \eqref{eq:2rdsa-l12}, it is easy to infer the following after ignoring the noise terms:
\begin{align*}
\widehat H_n(i,j)  & =  \nabla^2_{ij} f(x_n) + O(\delta_{n3^N}^2).
\end{align*}
The principal difference with feedback variants of 2SPSA/2RDSA is that the equation above implies that there are no zero-mean terms involving the perturbations that one needs to worry about in case of 2RDSA-DP. 

To substantiate the claim that the feedback terms are not necessary for 2RDSA-DP, we provide a rate result that parallels a corresponding result for 2SPSA with feedback \cite{spall-jacobian}. The proof given later is also much simpler than the corresponding version for 2SPSA, due to the fact that there are no extra terms involving perturbations to handle.
\begin{theorem}
\label{thm:2rdsa-dp-rate}
Assume (C9), (C10) and (C11), and also that the setting is noise-free. 
Let $b_n = b_0/n^r$, $n=1,2,\ldots,k$, where $1/2 < r < 1$ and $0 < b_0 \leq 1$. Let $H^*=\nabla^2 f(x)$ for any $x$, and $\Lambda_k = \overline H_k - H^*$. Then, we have 
\begin{align}
\textrm{trace}[\E (\Lambda_n \tr \Lambda_n)] = O(e^{-2b_0n^{1-r} / (1-r)}).
\label{eq:quad-bigo}
\end{align}
\end{theorem}
\begin{proof}
See Section \ref{sec:2rdsa-dp-proofs-rate}.
\end{proof}

\section{Convergence proofs}
\label{sec:proofs}
\subsection{Proofs for 1RDSA variants with deterministic perturbations}
\label{sec:1rdsa-dp-proofs}

\subsubsection{Proof of Lemma \ref{lemma:MNid}}  \label{sec:1rdsa-dp-proofs-MNid}
\begin{proof}\ \\[0.5ex]
\textbf{Case 1: Semi-lexicographic sequence-based perturbations}

We prove the claim by induction. Consider first the case of $N=1$. Now $D^{1}_{1} = \left[ 
\begin{array}{c} 
-1 \\ -1\\2
\end{array}
\right]$.
Note that $M_1 = (D^{1}_{1})\tr D^{1}_{1} = 6$.
Thus the result holds for $N=1$.

Suppose that the claim above holds for $N=k$, i.e., $M_k = 2\times 3^k \I_k$. Consider the case of $N=k+1$. We make the following observations:
\begin{enumerate}
	\item The size of each column $D^{i}_{k+1}, i=1,\ldots,k+1$ is $3^{k+1}$.
	\item The columns $D^{2}_{k+1},\ldots,D^{k+1}_{k+1}$ are obtained from $D^{1}_{k},\ldots,D^{k}_{k}$, respectively by concatenating the $D^{i}_{k}$ columns thrice, i.e., 
	\begin{align}
		D^{i}_{k+1} = \left[ 
		\begin{array}{c} 
			D^{i-1}_{k} \\ D^{i-1}_{k}\\D^{i-1}_{k}
		\end{array}
		\right], \quad i=2,\ldots,k+1. \label{eq:somelabel}
	\end{align}
	\item In the first column, i.e., $D^{1}_{k+1}$, the first $2\times 3^k$ elements are $-1$ and the remaining $3^k$ elements are $2$.
	\item 
	The off-diagonal terms in $M_{k+1}$ are all zero. This is argued as follows:
	For $i\ne j$, $i,j \in \{2,\ldots,k+1\}$, we have that $(D^{i}_{k+1})\tr D^{j}_{k+1} = 0$. This follows from \eqref{eq:somelabel} because $(D^{i}_{k+1})\tr D^{j}_{k+1} = 3 (D^{i-1}_{k})\tr D^{j-1}_{k} = 0$ (by induction hypothesis).
	
	For off-diagonal terms of the type $(D^{1}_{k+1})\tr D^{i}_{k+1}$, where $i\in \{2,\ldots,k+1\}$, we first re-write $D^{1}_{k+1}$ and $D^{i}_{k+1}$ as follows:
	\begin{align}
		D^{1}_{k+1} = \left[ 
		\begin{array}{c} 
			\bm{-1} \\ \bm{-1}\\\bm{2}
		\end{array}
		\right],\label{eq:somelabel1}
	\end{align}
	where $\bm{-1}$ and $\bm{2}$ are $k\times 1$ vectors with each entry $-1$ and $2$, respectively.
	Thus, from \eqref{eq:somelabel} and the foregoing, 
	\begin{align*}
		(D^{1}_{k+1})\tr D^{i}_{k+1} \!=\! - \bm{1}\tr D^{i-1}_{k} \!-\! \bm{1}\tr D^{i-1}_{k} \!+\! \bm{2}\tr D^{i-1}_{k}
		= 0.
	\end{align*}
	
	\item Finally, the diagonal terms of $M_{k+1}$ are $2\times 3^{k+1}$. This can be seen as follows: For $i=1,\ldots,k+1$,
	\begin{align*}
		(D^{i}_{k+1})\tr D^{i}_{k+1} &= \frac{2}{3} \times (-1)^2\times 3^{k+1} + \frac{1}{3} \times (2)^2\times 3^{k+1} \\
		&= 2\times 3^{k+1}.
	\end{align*}
	In the above, we have used the fact that $D^{i}_{k+1}$ is a $3^{k+1}$-dimensional vector with two-third entries $-1$ and the remaining one-third entries as $2$. 
\end{enumerate}
Thus, $M_{k+1} = 2\times 3^{k+1} \I_{k+1}$ and the claim follows for semi-lexicographic perturbations.

\textbf{Case 2: Permutation matrix-based perturbations}

Let $D_N$, the columns $D^{i}_{N}$, be the $N$-dimensional permutation matrix. Recall that 1RDSA-Perm-DP loops through the rows in $D_N$ and the columns $D^{i}_{N}$ in $D_N$ are in $N$-dimensions. 
It is well-known that $D_N^{-1} = D_N\tr$ for a permutation matrix $D_N$, cf. Proposition 2.7.21 of \cite{goodaire2013linear} for a detailed proof. Hence, the claim  that $M_N = \I_N$ follows for the case of permutation matrix based deterministic perturbation sequences.
\end{proof}

\subsubsection{Proof of Lemma \ref{lemma:1rdsa-dp-bias}} \label{sec:1rdsa-dp-proofs-bias}
\begin{proof}
We prove the claim for semi-lexicographic perturbations; the claim for the case of permutation matrix-based perturbations follows by an analogous argument.


By Taylor's series expansions, we obtain, a.s.,
\begin{align*}
	f(x_n \pm \delta_n d_m) & = f(x_n) \pm \delta_n d_m^T \nabla f(x_n) + \frac{\delta_n^2}{2} d_m^T \nabla^2 f(x_n) d_m 
	\pm \frac{\delta_n^3}{6} \nabla^3 f(\tilde{x}_n^+) (d_m \otimes d_m \otimes d_m)
\end{align*}
where $  \otimes $ denotes the Kronecker product and $ \tilde{x}_n^+ $ (respectively, $ \tilde{x}_n^- $) are on the line segment between $ x_n $ and $ (x_n + \delta_n d_m) $ (respectively, $ (x_n - \delta_n d_m) $). Then
\begin{align}
	&\E\left[\left.\widehat\nabla f(x_n)\right| \F_n \right]\nonumber\\
	&= \frac1{2\!\times\!3^N}  \sum\limits_{m=0}^{3^N-1} \E\left[\left.g_m\right| \F_n \right]\nonumber\\
	&=\frac1{2\!\times\!3^N}  \sum\limits_{m=0}^{3^N-1} \E\left[\left.\dfrac{d_m}{2\delta_{(n3^N+m)}} \times \right.\right.
	\left.\left.\left(f(x_n+\delta_{(n 3^{N} + m)} d_m) - f(x_n-\delta_{(n 3^{N} + m)} d_m)\right)\right| \F_n \right] \nonumber\\
	&= \frac1{2\!\times\!3^N}  \sum\limits_{m=0}^{3^N-1}   \left[d_m d_m\tr \nabla f(x_n) 
	 +  \frac{\delta_{(n3^N+m)}^2}{12} d_m(\nabla^3 f(\tilde{x}_n^+) + \nabla^3 f(\tilde{x}_n^-)) (d_m \otimes d_m \otimes d_m) \right]  \nonumber\\
	&=\nabla f(x_n)  + C_0 \delta_{n 3^{N}}^2.\label{eq:lex-unbaised}
\end{align}
The first term on the RHS of (\ref{eq:lex-unbaised}) follows from Lemma \ref{lemma:MNid}. Now, the $ l $th coordinate of the second term in the RHS of (\ref{eq:lex-unbaised}) can be upper-bounded as follows:
\begin{align*}
& \sum\limits_{m=0}^{3^N-1} 	\frac{\delta_{(n3^N+m)}^2}{12} d_m(\nabla^3 f(\tilde{x}_n^+) + \nabla^3 f(\tilde{x}_n^-)) (d_m \otimes d_m \otimes d_m) \\
& \leq  \frac{\alpha_0 \delta_{n3^N}^2}{6}  \sum\limits_{m=0}^{3^N-1} \sum_{i_1=1}^{N} \sum_{i_2=1}^{N} \sum_{i_3=1}^{N}  (d_m^l d_m^{i_1} d_m^{i_2} d_m^{i_3}) \\
& \leq \frac{\alpha_0 \delta_{n3^N}^2 N^3 (2\times3^N)^2}{6}
\end{align*}
The first inequality above follows from (A1), while the second inequality follows from the fact that  \\$\sum\limits_{m=0}^{3^N-1}  (d_m^l d_m^{i_1} d_m^{i_2} d_m^{i_3})$ is non-zero only if either $l=i_2$ and $i_1=i_3$ or vice-versa and in this case, we have
\[ \sum\limits_{m=0}^{3^N-1}  (d_m^l)^2 (d_m^{i_1})^2  = (2\times 3^N)^2.\]
The equality above can be easily inferred using induction arguments similar to proof of Lemma \ref{lemma:MNid} and we omit the details.
\end{proof}

\subsubsection{Proof of Theorem \ref{thm:1rdsa-strong-conv}} \label{sec:1rdsa-dp-proofs-strong-conv}
\begin{proof}
Follows in exactly the same manner as the proof of Theorem 2 of \cite{prashanth2017rdsa}, given the asymptotic unbiasedness result in Lemma \ref{lemma:1rdsa-dp-bias}.
\end{proof}

\subsubsection{Proof of Theorem \ref{thm:exp-bound}} \label{sec:1rdsa-dp-proofs-exp-bound}
\begin{proof}
Fixing $\delta_n$ for the inner loop and using arguments from the proof of Lemma \ref{lemma:1rdsa-dp-bias}, we obtain
\[ \widehat \nabla f(x_n) = \nabla f(x_n) + C_0 \delta_n^2 + \sum_{m=0}^{N-1} d_m \left(\frac{\xi_m^+ - \xi_m^-}{2\delta_n}\right). 
\]

Letting $\eta_n = \sum\limits_{m=0}^{N-1} d_m \left(\dfrac{\xi_m^+ - \xi_m^-}{2\delta_n}\right)$, the update in \eqref{eq:1rdsa} is equivalent to
\begin{align}
\label{eq:sgd-equiv}
 x_{n+1} = x_n - a_n \left[ \nabla f(x_n) + C_0 \delta_n^2 + \eta_n\right].
\end{align}
Since $\nabla f(x^*) = 0$, we have the following from the fundamental theorem of calculus: 
\[ \left(\int_0^1 \nabla^2 f(x^* + \lambda (x_n - x^*)) d\lambda \right)z_n = \nabla f(x_n).\]
Here $z_n \triangleq x_n - x^*$ denotes the optimization error at instant $n$ of 1RDSA-DP. 
Then, using \eqref{eq:sgd-equiv}, we have the following recursive update form for $z_n$:
\begin{align}
z_{n+1} = & (I-a_n J_n)z_n + a_n\left(C_0 \delta_n^2 + \eta_n\right)\nonumber\\
=& \tpi_n z_0 + \sum_{k=1}^{n}a_k\tpi_n\tpi_k^{-1}(C_0 \delta_k^2 + \eta_k), \label{eq:s33}
\end{align}
where $J_n := \int_0^1 \nabla^2 f(x^* + \lambda (x_n - x^*)) d\lambda$
and $\tpi_n:=\prod_{k=1}^{n}\left(I - a_k J_k\right)$.
A similar unrolling of a general stochastic approximation recursion can be found in \cite{frikha2012concentration}. However, our setting involves biased gradient estimates and the non-asymptotic bounds require a careful handling of the perturbation constant $\delta_n$, so that the overall convergence rate is of the order $O\left(1/\sqrt{n}\right)$. Moreover, we make all the constants explicit in the final bound.

Now, for the square of the error $\l z_{n+1} \r^2$, we use \eqref{eq:s33}
and Jensen's inequality to obtain
\begin{align}
(\E&\l z_{n+1}\r)^2
\le \E (\langle z_n, z_n \rangle )\nonumber\\
=&  \E \bigg[ \l \tpi_n z_0 \r^2 
				+ \l\sum_{k=1}^{n}a_k \tpi_n\tpi_k^{-1}C_0 \delta_k^2\r^2\
+ \l\sum_{k=1}^{n}a_k\tpi_n\tpi_k^{-1}\eta_k\r^2 + \left\langle \tpi_n z_0,
					\sum_{k=1}^{n}a_k \tpi_n\tpi_k^{-1}C_0\delta_k^2 \right\rangle \nonumber\\
				&+
	\left
				\langle \tpi_n z_0,
						\sum_{k=1}^{n}a_k \tpi_n\tpi_k^{-1}\eta_k \right\rangle 
+ \left\langle  \sum_{k=1}^{n}a_k \tpi_n\tpi_k^{-1}C_0\delta_k^2,
						\sum_{k=1}^{n}a_k \tpi_n\tpi_k^{-1}\eta_k \right\rangle
																						\bigg] \nonumber\\
&	\le  2\l \tpi_n z_0 \r^2
				+ 3\sum_{k=1}^{n}a_k^2 \l\tpi_n\tpi_k^{-1}\r^2 C_0^2 \delta_k^4
				+ 2\sum_{k=1}^{n}a_k^2 \l\tpi_n\tpi_k^{-1}\r^2\E\l\eta_k\r^2.
\label{eq:exp-inter}
\end{align}
For the last inequality, we have used the fact that $\eta_k$ is zero-mean (see (A2)) in order to ignore a cross term. For the other two cross terms, we have used Cauchy-Schwarz to conclude $\langle a, b \rangle \le \max( \l a\r^2, \l b \r^2)$ and hence, the first and last square terms can appear at most twice, while the second square term can appear at most thrice.

The second moment of the noise factor $\eta_n$ can be bounded as follows: 
\begin{align}
 \E \l \eta_n \r^2 &\le \sum\limits_{m=0}^{N-1} \l d_m \r^2 \E \left(\dfrac{\xi_m^+ - \xi_m^-}{2\delta_n}\right)^2
\le \frac{ N \alpha_1}{2 \delta_n^2}, \label{eq:noise-bd}
\end{align}
where we have used the fact that for permutation matrix-based perturbation $\l d_m \r = 1$ and $\E (\xi_m^+ - \xi_m^-)^2 \le 2 \alpha_1$ from assumption (A3).

The term $\ml \tpi_n \tpi_k^{-1}\mr$ is bounded as follows:
\begin{align}
\ml \tpi_n \tpi_k^{-1}\mr =& \ml\prod_{j=k+1}^{n}\left(I - a_j J_j\right)\mr\nonumber\\
   =&\prod_{j=k+1}^{n}\ml (1-a_j\mu)I - a_j(J_j- \mu I) \mr\nonumber\\
		\le&\prod_{j=k+1}^{n}\ml (1-a_j\mu)I\mr \le \prod_{j=k+1}^{n} (1-a_j\mu)\nonumber\\
		\le&\e\left(-\mu(\Gamma_n - \Gamma_k)\right).\label{eq:t12}
\end{align}
The second inequality above follows by observing that $\l I - a_n J_n \r  \le \e(-\mu a_n)$, since $\nabla^2 f(x) - \mu I$ is positive semi-definite owing to strong-convexity of $f$ (see (A1')).

The main claim now follows by plugging in the bounds in \eqref{eq:noise-bd} and \eqref{eq:t12} into \eqref{eq:exp-inter}.
\end{proof}

\subsection{Proof of Theorem \ref{thm:1rdsa-dp-rate}} \label{sec:1rdsa-dp-proofs-dp-rate}
\begin{proof}
We bound each of the error terms on the RHS of \eqref{eq:expectation-bound} separately.
For bounding the initial error, we use the following inequality:
\[\e(-\mu\Gamma_n) \le \e(-\mu c \ln n) \le n^{-\mu c}.\]
In arriving at the bound above, we have compared a sum with an integral.

Substituting $a_k = c/k$ and $\delta_k = \delta_0/k^{\delta}$ into the bias error term in \eqref{eq:expectation-bound}, we obtain
\begin{align*}
\sum\limits_{k=1}^{n}a^2_{k}
		  							\e(-2\mu(\Gamma_n - \Gamma_{k})) C_0^2 \delta_k^4 
		  							&\le 
\sum_{k=1}^n \frac{c^2}{k^2} n^{-2\mu c} k^{2 \mu c} C_0^2 \dfrac{\delta_0^4}{k^{4\delta}} \\
&\le  c^2 n^{-2 \mu c} C_0^2 \delta_0^4 \sum_{k=1}^n k^{2 \mu c - 4\delta -2} \\
&\le  \dfrac{c^2  C_0^2 \delta_0^4}{(2\mu c - 4\delta - 1)} n^{-1-4\delta}.
\end{align*}

Along similar lines, the sampling error term in \eqref{eq:expectation-bound} can be upper-bounded as follows:
\begin{align*}
\sum\limits_{k=1}^{n}a^2_{k}
		  							\e(-2\mu(\Gamma_n - \Gamma_{k})) \dfrac{C_1}{\delta_k^2} 
&\le  \dfrac{c^2  C_1}{\delta_0^2 (2\mu c - 4\delta - 1)} n^{-1+2\delta}.
\end{align*}
The claim follows by combining the bounds derived above on each of the error terms on the RHS of \eqref{eq:expectation-bound}.
\end{proof}

\subsection{Proofs for 2RDSA variants with deterministic perturbations}
\label{sec:2rdsa-dp-proofs}

\subsubsection{Proof of Lemma \ref{lemma:2rdsa-dp-bias}} \label{sec:2rdsa-dp-proofs-bias}

\begin{proof}\ \\
\textbf{Case 1: Semi-lexicographic sequence-based perturbations}

As in the proof of Lemma 4 in \cite{prashanth2017rdsa}, we employ Taylor's series expansions of $f(\cdot)$ at $x_n \pm \delta_n d_n$ to obtain
\begin{align*}
&\dfrac{f(x_n+\delta_{(n3^N + m)} d_m) + f(x_n-\delta_{(n3^N + m)} d_m) - 2 f(x_n)}{\delta_{(n3^N + m)}^2}\\
 =& d_m\tr \nabla^2 f(x_n) d_m +  O(\delta_{(n3^N + m)}^2)\\
= & \sum\limits_{i=1}^N\sum\limits_{j=1}^N d_m^i d_m^j \nabla^2_{ij} f(x_n) + O(\delta_{(n3^N + m)}^2)\\
=& \sum\limits_{i=1}^N (d_m^i)^2 \nabla^2_{ii} f(x_n) + 2\sum\limits_{i=1}^{N-1}\sum\limits_{j=i+1}^N d_m^i d_m^j \nabla^2_{ij} f(x_n)
 + O(\delta_{(n3^N + m)}^2).\stepcounter{equation}\tag{\theequation}\label{eq:2rdsa-taylor}
\end{align*}
Recall that the Hessian estimate is given by
\begin{align*}
\widehat{H_n} &=  \dfrac{1}{(2\times 3^N)^2} \sum\limits_{m=0}^{3^N-1} M_m \left(\dfrac{y_m^+ + y_m^- - 2 y_m}{\delta_n^2}\right), \textrm{ where }
\end{align*}
$M_m$ is as defined in Algorithm \ref{alg:2rdsa-dp}.
For the sake of simplicity, we ignore the zero-mean noise term  
$\xi_m^+ + \xi_m^- - 2\xi_m$ temporarily and analyze the following product inside the Hessian estimate:
\begin{align}
&\dfrac{1}{(2\times 3^N)^2} \sum\limits_{m=0}^{3^N-1} M_m \left(\sum\limits_{i=1}^{N} (d_m^i)^2 \nabla^2_{ii} f(x_n) + 2\sum\limits_{i=1}^{N-1}\sum\limits_{j=i+1}^N d_m^i d_m^j \nabla^2_{ij} f(x_n) )\right). \label{eq:h1}
\end{align}

\subsubsection{Off-diagonal terms in \eqref{eq:h1}}

We now consider the $(k,l)$th term in \eqref{eq:h1}: Assume without loss of generality, that $k<l$. Then,
\begin{align}
& \dfrac{1}{(2\times 3^N)^2} \sum\limits_{m=0}^{3^N-1}\left[d_m^k d_m^l   \left(\sum\limits_{i=1}^N (d_m^i)^2 \nabla^2_{ii} f(x_n) \right.\right. 
\left.\left.+ 2\sum\limits_{i=1}^{N-1}\sum\limits_{j=i+1}^N d_m^i d_m^j \nabla^2_{ij} f(x_n)\right) \right]\nonumber\\
=& \dfrac{1}{(2\times 3^N)^2} \sum\limits_{i=1}^N \nabla^2_{ii} f(x_n) \sum\limits_{m=0}^{3^N-1}  \left(d_m^k d_m^l (d_m^i)^2 \right) 
+ \dfrac{1}{(2\times 3^N)^2}\sum\limits_{i=1}^{N-1}\sum\limits_{j=i+1}^N  \nabla^2_{ij} f(x_n) \sum\limits_{m=0}^{3^N-1}  \left(d_m^k d_m^l d_m^i d_m^j\right)  \label{eq:crossh}\\
= & \nabla^2_{kl} f(x_n).\label{eq:offidag}
\end{align}
The last equality follows from the fact that the first term in \eqref{eq:crossh} is $0$ since $k\ne l$ and 
\[\sum\limits_{m=0}^{3^N-1}  \left(d_m^k d_m^l (d_m^i)^2 \right)=0 \textrm{ for any } i,\]
while the second term in \eqref{eq:crossh} can be seen to be equal to $\nabla^2_{kl} f(x_n)$ 
using induction arguments similar to proof of Lemma \ref{lemma:1rdsa-dp-bias}.
 
\subsubsection{Diagonal terms in \eqref{eq:h1}}

Consider the $lth$ diagonal term inside the conditional expectation in \eqref{eq:h1}:
\begin{align}
& \dfrac{\kappa}{(2\times 3^N)^2} \sum\limits_{m=0}^{3^N-1}  \left((d_m^l)^2-(2\times 3^N)\right)  \times \left(\sum\limits_{i=1}^N (d_m^i)^2 \nabla^2_{ii} f(x_n)
 + 2\sum\limits_{i=1}^{N-1}\sum\limits_{j=i+1}^N d_m^i d_m^j \nabla^2_{ij} f(x_n)\right)\nonumber\\
=& \dfrac{\kappa}{(2\times 3^N)^2} \sum\limits_{m=0}^{3^N-1} (d_m^l)^2 \sum\limits_{i=1}^N (d_m^i)^2 \nabla^2_{ii} f(x_n) 
 + \dfrac{2\kappa}{(2\times 3^N)^2} \sum\limits_{m=0}^{3^N-1} (d_m^l)^2\sum\limits_{i=1}^{N-1}\sum\limits_{j=i+1}^N d_m^i d_m^j \nabla^2_{ij} f(x_n)\nonumber\\
& - \dfrac{\kappa}{(2\times 3^N)} \sum\limits_{m=0}^{3^N-1} \sum\limits_{i=1}^N (d_m^i)^2 \nabla^2_{ii} f(x_n) 
- \dfrac{2\kappa}{(2\times 3^N)} \sum\limits_{m=0}^{3^N-1}\sum\limits_{i=1}^{N-1}\sum\limits_{j=i+1}^N d_m^i d_m^j \nabla^2_{ij} f(x_n).\label{eq:h2}
\end{align}
Now, we analyze each of the four terms on the RHS above.

The first term on the RHS of \eqref{eq:h2} can be simplified as follows:
\begin{align*}
&\dfrac{\kappa}{(2\times 3^N)^2} \sum\limits_{m=0}^{3^N-1} (d_m^l)^2 \sum\limits_{i=1}^N (d_m^i)^2 \nabla^2_{ii} f(x_n)\\
= & \dfrac{\kappa}{(2\times 3^N)^2} \left(\sum\limits_{m=0}^{3^N-1}(d_m^l)^4 \nabla^2_{ll} f(x_n) \right.
\left. + \sum\limits_{m=0}^{3^N-1}\sum\limits_{i=1,i\ne l}^N (d_m^l)^2(d_m^i)^2 \nabla^2_{ii} f(x_n) \right)\\
= & \dfrac{\kappa}{(2\times 3^N)^2}  \Bigg( (2\times 3^{N+1}) \nabla^2_{ll} f(x_n)
+ (2\times 3^N)^2\sum\limits_{i=1,i\ne l}^N  \nabla^2_{ii} f(x_n) \Bigg).
\end{align*} 

For the second equality above, we have used the fact that $\sum\limits_{m=0}^{3^N-1}(d_m^l)^4 = 2\times 3^{N+1}$ (easy to infer this claim along the lines of point (5) in the proof of Lemma \ref{lemma:MNid}) and $\sum\limits_{m=0}^{3^N-1} (d_m^l)^2(d_m^i)^2  = (2\times 3^N)^2$, $\forall l \ne i$.

The second term in \eqref{eq:h2} is zero because
$\sum\limits_{m=0}^{3^N-1}  \left(d_m^i d_m^j (d_m^l)^2 \right)=0$ for any $l$ and $i\ne j$ -- a claim that can be easily proved using an induction argument.

The third term in \eqref{eq:h2} without the negative sign can be simplified as follows: 
\begin{align*}
\dfrac{\kappa}{(2\times 3^N)} \sum\limits_{m=0}^{3^N-1} \sum\limits_{i=1}^N (d_m^i)^2 \nabla^2_{ii} f(x_n)
&=  \dfrac{\kappa}{(2\times 3^N)} \sum\limits_{i=1}^N \sum\limits_{m=0}^{3^N-1} (d_m^i)^2 \nabla^2_{ii} f(x_n) \\
&=   \kappa\sum\limits_{i=1}^N \nabla^2_{ii} f(x_n), \text{ a.s.}
\end{align*} 
Combining the above followed by some algebra, we obtain 
\begin{align*}
&\dfrac{\kappa}{(2\times 3^N)^2}  \E\left[\left. \left((d_m^l)^2-(2\times 3^N)\right) \left(\sum\limits_{i=1}^N (d_m^i)^2 \nabla^2_{ii} f(x_n) \right.\right.\right.
\left.\left.\left. + 2\sum\limits_{i=1}^{N-1}\sum\limits_{j=i+1}^N d_m^i d_m^j \nabla^2_{ij} f(x_n)\right)\right| \F_n\right]\\
&= \nabla^2_{ll} f(x_n). \stepcounter{equation}\tag{\theequation}\label{eq:diag}
\end{align*}
The fourth term in \eqref{eq:h2} is zero from Lemma \ref{lemma:MNid}.

Denote the product in \eqref{eq:h1} by $(A)$ and the $(i,j)$th term there by $(A)_{i,j}$. Then,
\begin{align*}
\E[\widehat H_n(i,j) \mid \F_n] & =  \E\left[ (A) \mid \F_n\right] + \E[\xi_n^+ + \xi_n^- - 2\xi_n \mid \F_n] \stepcounter{equation}\tag{\theequation}\label{eq:2rdsa-l12}\\
&= \nabla^2_{ij} f(x_n) + O(\delta_{(n3^N)}^2).
\end{align*}
The last equality above follows from (C3), while the term involving the factor $(A)$ reduces to the true Hessian with a bias of $O(\delta_{(n3^N)}^2)$ due to \eqref{eq:offidag} and \eqref{eq:diag}.

The main claim follows.

\textbf{Case 2: Permutation matrix-based perturbations}

From the first step involving Taylor series expansions in Case 1, we have 
\begin{align*}
&\dfrac{f(x_n+\delta_{(n N + m)} d_m) + f(x_n-\delta_{(n  N + m)} d_m) - 2 f(x_n)}{\delta_{(n N + m)}^2}\\
=& \sum\limits_{i=1}^N (d_m^i)^2 \nabla^2_{ii} f(x_n) + 2\sum\limits_{i=1}^{N-1}\sum\limits_{j=i+1}^N d_m^i d_m^j \nabla^2_{ij} f(x_n)
+ O(\delta_{(n N + m)}^2).\stepcounter{equation}\tag{\theequation}\label{eq:2rdsa-taylor1}
\end{align*}
Using the facts that 
$\sum\limits_{m=0}^{N-1} \sum\limits_{i=1}^N d_m^i d_m^j =0$ for $i\ne j$ and  $\sum\limits_{m=0}^{N-1} (d_m^i)^2=1$ for any $i$, we obtain
\begin{align*}
& \sum\limits_{m=0}^{N-1} \left[\dfrac{f(x_n+\delta_{(n N + m)} d_m) + f(x_n-\delta_{(n N + m)} d_m) - 2 f(x_n)}{\delta_{(n N + m)}^2}\right]\\
& = \sum\limits_{i=1}^N\sum\limits_{m=0}^{N-1} (d_m^i)^2 \nabla^2_{ii} f(x_n)
+ 2\sum\limits_{i=1}^{N-1}\sum\limits_{j=i+1}^N \nabla^2_{ij} f(x_n)\sum\limits_{m=0}^{N-1}d_m^i d_m^j 
+ O(\delta_{(n N + m)}^2)\stepcounter{equation}\tag{\theequation}\label{eq:2rdsa-h12}\\
&=  \nabla^2_{ii} f(x_n) + O(\delta_{(n N)}^2).
\end{align*}

Denote the term on LHS in \eqref{eq:2rdsa-h12} by $(B)$, and following an argument similar to that used in simplifying the noise terms in \eqref{eq:2rdsa-l12}, we obtain
\begin{align*}
\E[\widehat H_n(i,i) \mid \F_n] & =  \E\left[ (B) \mid \F_n\right] + \E[\xi_n^+ + \xi_n^- - 2\xi_n \mid \F_n]\\
& = \nabla^2_{ii} f(x_n) + O(\delta_{(n N)}^2).
\end{align*}
Hence proved.
\end{proof}

\subsubsection{Proof of Theorem \ref{thm:2rdsa-dp-H}} \label{sec:2rdsa-dp-proofs-H}
\begin{proof}
Follows in a similar manner as that of the proofs of Theorem 5 and 6 in \cite{prashanth2017rdsa}.
\end{proof}

\subsubsection{Proof of Theorem \ref{thm:2rdsa-dp-rate}} \label{sec:2rdsa-dp-proofs-rate}
\begin{proof}
From the quadratic assumption, we have that $H(x)$ is a constant, independent of $x$.
From the proof leading up to \eqref{eq:2rdsa-l12}, we have that 
$\widehat H_n = \nabla^2 H(x_n) = H^*$. 
Now, we follow the technique from \cite{spall-jacobian} to derive the main claim. 

Notice that
\[ \Lambda_k = (1 - b_k)\Lambda_{k-1} + b_k(\overline H_k - H^*) = (1 - b_k)\Lambda_{k-1}. \]
Unrolling the recursion above, we obtain
\begin{align*}
\Lambda_n &= \left[\prod_{k=1}^{n}(1 - b_k)\right] \Lambda_0.
\end{align*}
Using the above, we can arrive at a simpler representation for the $ \textrm{trace} \left[\E(\Lambda_n^T \Lambda_n)\right]  $ as follows:
\begin{align*}
\E(\Lambda_n^T \Lambda_n) &= \left[\prod_{k=1}^{n}(1 - b_k)\right]^2 \E(\Lambda_0^T \Lambda_0), \textrm{ leading to }\\
\textrm{trace} \left[\E(\Lambda_n^T \Lambda_n)\right] &= \left[\prod_{k=1}^{n}(1 - b_k)\right]^2 \textrm{trace} \left[ \E(\Lambda_0^T \Lambda_0)\right].
\end{align*}
Simplifying further,  using the fact that $ 1-b_k = e^{-b_k}(1-O(b_k^2))$, with the $ O(b_k^2) $ being strictly positive as  $ e^{-b_k} $ is convex, we have
\begin{align}
\textrm{trace}\left[\E(\Lambda_n^T \Lambda_n)\right] &= e^{-2 b_{sum}(1,n)} c_{0n} \textrm{trace} \left[ \E(\Lambda_0^T \Lambda_0)\right],\label{eq:tr12}
\end{align}
where $ b_{sum}(i,j) = \sum_{k=i}^{j}b_k $
and $ c_{kn} = \left[\prod_{i=k+1}^{n} (1 - {O}(b_i^2))  \right]^2, k \geq 0$ and $c_{nn} = 1 $. Since $ 0 < b_k < 1$,  $\forall k \ge 2$, and $ r > 1/2 $, the $ c_{kn} $ are uniformly bounded in magnitude. Further, 
\[ b_{sum}(1,n) \geq \int_{1}^{n+1}\frac{b_0}{x^r} dx   \geq \left(\frac{b_0}{1-r} \right) \left( n^{1-r} - 1\right). \]
Using the bound derived above in \eqref{eq:tr12}, we obtain
\begin{align*}
\textrm{trace} \left[\E(\Lambda_n^T \Lambda_n)\right] &\le e^{-2b_0 n^{1-r} / (1 - r)} e^{2b_0} c_{0n} \textrm{trace} \left[ \E(\Lambda_0^T \Lambda_0)\right].
\end{align*}
The main claim follows.
\end{proof}

\section{Numerical Experiments}
\label{sec:expts}
\pgfplotsset{every x tick label/.append style={font=\scriptsize}}
\usetikzlibrary{patterns}
\usetikzlibrary{pgfplots.groupplots}

\subsection[Implementation] {Implementation\footnote{The implementation is available at \url{https://github.com/prashla/RDSA/archive/master.zip}.}}
We consider the following problem:
\begin{align}
	\min_{x} \; \E_{\xi} \left[ F(x,\xi) \right],
	\label{eq:obj}
\end{align}
where $ F(x,\xi) $ is the sample observation of the objective function $ f(x) $ corrupted with zero mean noise $ \xi $.
In particular, the noise is $ [x^T,1]\xi $, where $ \xi $ is a multivariate Gaussian distribution with mean zero and covariance $  \sigma^2 \I_{N+1} $. A similar noise structure has been used earlier in the implementation of both RDSA and SPSA algorithms in \cite{prashanth2017rdsa} and \cite{spall2000adaptive}, respectively. For all the experiments, we consider two different settings of noise: (i) low noise with $ \sigma =  0.001$; and (ii) high noise with $ \sigma = 0.1 $.

We consider three different functional forms for $ F(x,\xi) $, namely quadratic, fourth-order and Rastrigin, respectively, in both $ N = 5 $ and $ N = 10 $ dimensions, for evaluating our algorithms. Before describing the example functions, we present the details about the algorithms implemented.

We implement the first-order and second-order algorithms proposed in this paper and compare them with several baselines that are based on the simultaneous perturbation method. 

The first-order algorithms implemented include 1RDSA-Lex-DP, 1RDSA-Perm-DP and 1RDSA-KW-DP - the three deterministic perturbation variants of 1RDSA (see Section \ref{sec:1rdsa-dp} for a detailed description). We compare  these algorithms with the following baselines: RDSA with uniform and asymmetric Bernoulli perturbations proposed in \cite{prashanth2017rdsa} and henceforth referred to as 1RDSA-Unif and 1RDSA-AsymBer, respectively, and SPSA with Bernoulli perturbations, henceforth referred to as 1SPSA.

The second-order algorithms implemented include 2RDSA-Lex-DP and 2RDSA-Perm-DP - the two deterministic perturbation variants of 2RDSA (see Section \ref{sec:2rdsa-dp} for a detailed description). We compare  these algorithms with the following baselines: second-order RDSA with uniform and asymmetric Bernoulli perturbations proposed in \cite{prashanth2017rdsa} and henceforth referred to as 2RDSA-Unif and 2RDSA-AsymBer, respectively, and second-order SPSA with Bernoulli perturbations, henceforth referred to as 2SPSA.

The settings of the parameters $ \delta_n  $ and $ a_n $ for both first- and second-order algorithms are listed in Table \ref{tab:parameter-setting}, and a similar setting has been used in implementation of both RDSA and SPSA algorithms in \cite{prashanth2017rdsa} and \cite{spall2000adaptive}, 
respectively. 
The distribution parameter for RDSA variants is set as follows: $ \eta = 1 $ for RDSA-Unif, $ \epsilon = 0.0001 $ for 1RDSA-AsymBer, and $ \epsilon = 1  $ for 2RDSA-AsymBer, and a similar setting has been used earlier for RDSA implementation in \cite{prashanth2017rdsa}.
Each coordinate of the parameter is projected onto the set $ [-2.048,2.047] $, which helps to keep the iterates stable. All results are averages over 50 independent runs.

\begin{table}[H]
	\caption{Step-size and perturbation constant parameter settings, for first and second order algorithms.}
	\label{tab:parameter-setting}
	\centering
	\begin{tabular}{|c|c|c|}
		\toprule
		Algorithms  &$ \delta_n  $ & $ a_n $\\
		\midrule
		First-order &$ 1.9/n^{0.101}$ & $1/(n + 50) $\\
		Second-order &$ 3.8/n^{0.101}  $ & $ 1/n^{0.6} $\\
		\bottomrule
	\end{tabular}
\end{table}

To compare the algorithms' performance, we use parameter error as the performance metric. For a given simulaton budget, the parameter error measures the distance between the final iterate obtained after the final update iteration and the optimum parameter $x^*$. More precisely, we use the following form for the parameter error, after suitable normalization:
\[    \text{Parameter error} = \frac{|| x_{{\tau}} - x^* ||^2}{|| x_{0} - x^* ||^2},\]
where $ {\tau} $ is the number of times $ x $ is updated until the end of the simulation. Notice that $ {\tau} $ varies with the algorithm and the number of function measurements. For example, with a budget of 5000 measurements for $ N = 10 $, $ {\tau} = 250 $ for 1RDSA-Perm-DP and 1RDSA-KW-DP as they use $ 2N $ measurements per iteration. Further, $ \tau=2500 $ for 1SPSA as well as for both variants of 1RDSA, as they use two measurements per iteration. Notice that 1RDSA-Lex-DP does not make much progress under low simulation budgets, as it requires $  2 \times 3^N$ measurements per iteration.
On the other hand, for second-order algorithms, an initial 20\% of the measurements were used up by the corresponding first-order algorithm for initialization. Thus, for $ N = 10 $ and a budget of 5000 measurements, the initial 1000 measurements are used for the first-order algorithm and the remaining 4000 are used by the second-order algorithm. As a consequence of the simulation budget split, the number of update iterations $\tau = 4000/30 \approx 133 $ for 2RDSA-Perm-DP, $ 4000/4 = 1000 $ for 2SPSA, and $ 4000/3 \approx 1333 $ for 2RDSA algorithms. The 2RDSA-Lex-DP algorithm does not output a meaningful parameter with a low simulation budget, owing to its high inner loop length. The difference here is due to the fact that 2RDSA-Perm-DP uses $ 3N $ measurements per iteration, while 2RDSA-Lex-DP needs $  3 \times 3^N$, 2RDSA needs 3 and 2SPSA needs 4.

\begin{figure*}[t]
	\centering
	\hspace{-2em}
	\begin{minipage}{\textwidth}
		\begin{tabular}{cc}			
			\begin{subfigure}{0.5\textwidth}
				\centering
				\label{fig:first-order-quadratic}
				\hspace{-2em} 
				\tabl{c}{\scalebox{0.95}{
						\begin{tikzpicture}
						\begin{axis}[
						ylabel={Parameter error $ \cdot 10^-6 $},
						enlargelimits=0.15,
						legend pos=north west,
						y filter/.code={\pgfmathparse{#1*10^6}\pgfmathresult},
						symbolic x coords={2RDSA-Lex-DP,2RDSA-AsymBer,2RDSA-Unif,2SPSA},
						xtick=data,
						x tick label style={rotate=45,anchor=east},
						ybar=5pt,
						bar width=9pt,
						]	
						
						
					\addplot
				table[x=savefinal1,y=savefinal2,col sep=comma] {pgfplots_data/second_order_quad_sigma_0.001.dat}; 
						
						
						
						\end{axis}
						\end{tikzpicture}
				}}
				\caption{$ \sigma = 0.001 $}
			\end{subfigure}
			&
			\begin{subfigure}{0.5\textwidth}
				\centering
				\label{fig:second-order-quadratic}
				\hspace{-2em} 
				\tabl{c}{\scalebox{0.95}{
						\begin{tikzpicture}
						\begin{axis}[
						ylabel=Parameter error $\cdot 10^-2 $,
						enlargelimits=0.15,
						y filter/.code={\pgfmathparse{#1*10^2}\pgfmathresult},
						symbolic x coords={2RDSA-Perm-DP,2RDSA-Lex-DP,2RDSA-AsymBer,2RDSA-Unif,2SPSA},
						xtick=data,
						x tick label style={rotate=45,anchor=east},
							ybar=5pt,
						bar width=9pt,
						legend pos=north west,
						]
						
						
						
						\addplot
						table[x=savefinal1,y=savefinal2,col sep=comma] {pgfplots_data/second_order_quad_sigma_0.1.dat}; 
						
						
						
						\end{axis}
						\end{tikzpicture}
				}}
				\caption{$ \sigma = 0.1 $}
			\end{subfigure}
		\end{tabular}
        \caption{Parameter error for various second-order algorithms under the quadratic objective \eqref{eq:quad-fun} for a five-dimensional problem with a simulation budget of 50000 and $\sigma = 0.001 $ and $0.1$.}
		\label{fig:first-order-bar}
	\end{minipage}
\end{figure*}
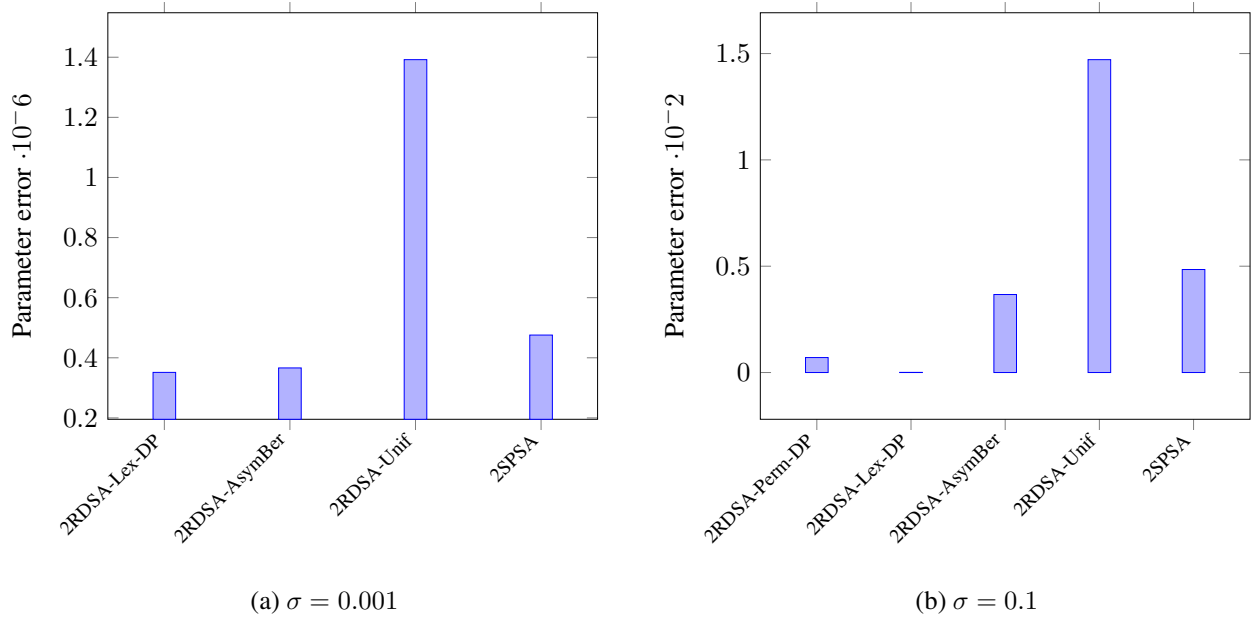

\subsection{Example 1: Quadratic objective}
Let $ A $ be such that $ N A $ is an $ N \times N $ upper triangular matrix with $ a_{ij} = 1 $ for $ i \le j $ and let $ b $ be an $ N $-dimensional vector of ones. Then the quadratic objective function is defined as follows:
\begin{align}
	F(x,\xi) = x^TAx + b^Tx + \xi.
	\label{eq:quad-fun}
\end{align}
The initial point $ x_0 $ is set to the $ N $-dimensional vector of ones and the optimal point $ x^*$ is dimension dependent. For instance, with $ N = 10 $, the optimal point $ x^*$, is the $10$-dimensional vector of $ -0.9091 $ for the choice of $A$ and $b$ described earlier. Note that $f(x^*) = \E_{\xi}[ F(x^*,\xi)]= -4.55$.

Figures \ref{fig:first-order-bar} and \ref{fig:first-order-bar1} present the parameter error for the first-order and second-order algorithms under the quadratic objective \eqref{eq:quad-fun} for dimension 5 and $ \sigma = 0.001$ and $0.1 $.

Among the first-order algorithms, for both settings of noise, 1RDSA-Perm-DP and 1RDSA-KW-DP exhibited similar performance and outperformed the other algorithms. The parameter error in 1RDSA-Perm-DP and 1RDSA-KW-DP is of the order of $ 10^{-5} $, while for the others, the same is of the order of $ 10^{-3} $. 1RDSA-Lex-DP (not shown in the figure) showed a parameter error that was an order of magnitude higher than the other algorithms.

Among the second-order algorithms, for both settings of noise, 2RDSA-Lex-DP exhibited the best performance. Furthermore, it is interesting to see that for both settings of noise 2RDSA-Perm-DP, and 2RDSA-Lex-DP gave consistent performance with parameter error of the order of $ 10^{-4} $ and $ 10^{-7} $, respectively, while the parameter error of 2RDSA-AsymBer, 2RDSA-Unif, and 2SPSA increased with noise. Further, the benefit of using second-order algorithms is more noticeable under the low noise setting. 

For the low noise setting, the parameter error of 2RDSA-Perm-DP is not shown in the figure, for the sake of readability in comparing the errors of the other algorithms.




\begin{figure}[H]
	\centering
	\begin{tikzpicture}
\begin{axis}[
ylabel={Parameter error $ \cdot 10^-3 $},
enlargelimits=0.15,
legend pos=north west,
						y filter/.code={\pgfmathparse{#1*1000}\pgfmathresult},
symbolic x coords={1RDSA-Perm-DP,1RDSA-KW-DP,1RDSA-AsymBer,1RDSA-Unif,1SPSA},
xtick=data,
x tick label style={rotate=45,anchor=east},
ybar=5pt,
bar width=9pt,
]	


\addplot
table[x=savefinal1,y=savefinal2,col sep=comma] {pgfplots_data/first_order_quad_sigma_0.001.dat}; 

\addplot
table[x=savefinal1,y=savefinal2,col sep=comma] {pgfplots_data/first_order_quad_sigma_0.1.dat}; 

\legend{$ \sigma = 0.001 $, $\sigma = 0.1$}

\end{axis}
\end{tikzpicture}
	\caption{Parameter error for various first-order algorithms under the quadratic objective \eqref{eq:quad-fun} for a five-dimensional problem with a simulation budget of 50000 and $\sigma = 0.001 $ and $0.1$.}
	\label{fig:first-order-bar1}
\end{figure}
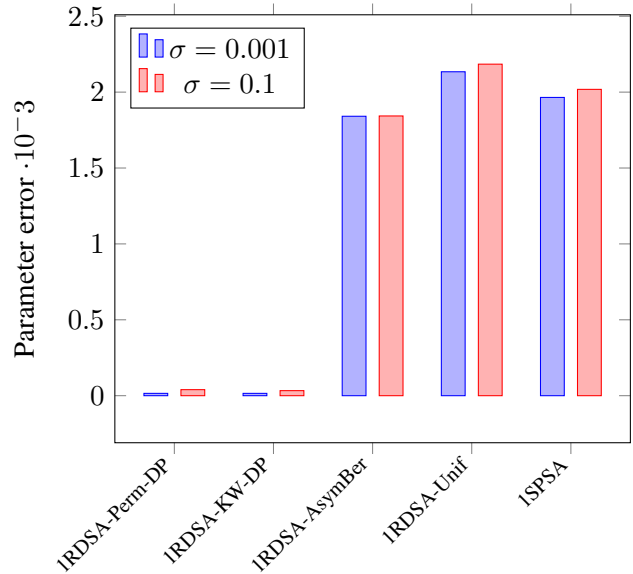

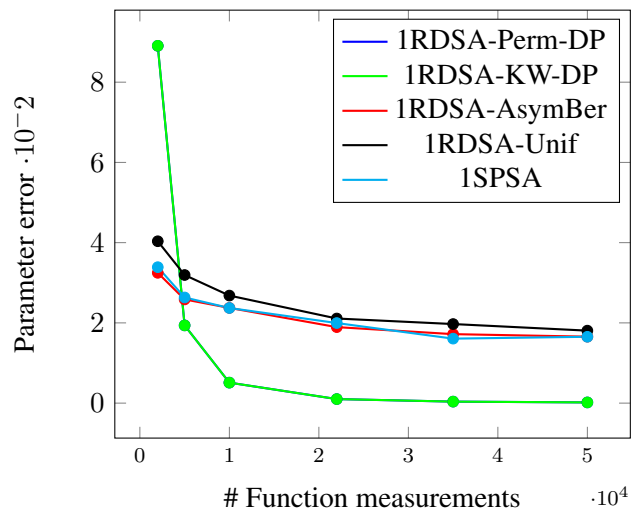
\begin{figure}[H]
		\centering
		\pgfplotsset{
			legend entry/.initial=,
			every axis plot post/.code={%
				\pgfkeysgetvalue{/pgfplots/legend entry}\tempValue
				\ifx\tempValue\empty
				\pgfkeysalso{/pgfplots/forget plot}%
				\else
				\expandafter\addlegendentry\expandafter{\tempValue}%
				\fi
			},
		}
	\begin{tikzpicture} 
	\begin{axis}[
	ylabel=Parameter error $ \cdot 10^-2 $,
	xlabel = \# Function measurements,
	scaled y ticks = false,
	y filter/.code={\pgfmathparse{#1*100}\pgfmathresult},
	]
	\addplot[clip marker paths=true,color = blue,mark = none,thick,legend entry=1RDSA-Perm-DP,] table [x index=0, y index=1] {pgfplots_data/results_onerdsa_perm_dp_dim_10_type_1.txt};
	\addplot[only marks, blue, mark=*, mark options={blue}, error bars/.cd,y dir=both,  y explicit] table [x index=0, y index=1, y error index=2]{pgfplots_data/results_onerdsa_perm_dp_dim_10_type_1.txt};
	
	\addplot[clip marker paths=true,color = green,mark = none,thick,legend entry=1RDSA-KW-DP,] table [x index=0, y index=1] {pgfplots_data/results_onerdsa_kw_dp_dim_10_type_1.txt};
	\addplot[only marks, green, mark=*, mark options={green}, error bars/.cd,y dir=both,  y explicit] table [x index=0, y index=1, y error index=2]{pgfplots_data/results_onerdsa_kw_dp_dim_10_type_1.txt};
	
	\addplot[clip marker paths=true,color = red,mark = none,thick,legend entry=1RDSA-AsymBer,] table [x index=0, y index=1] {pgfplots_data/results_onerdsa_asymber_dim_10_type_1.txt};
	\addplot[only marks, red, mark=*, mark options={red}, error bars/.cd,y dir=both,  y explicit] table [x index=0, y index=1, y error index=2]{pgfplots_data/results_onerdsa_asymber_dim_10_type_1.txt};

	\addplot[clip marker paths=true,color = black,mark = none,thick,legend entry=1RDSA-Unif,] table [x index=0, y index=1] {pgfplots_data/results_onerdsa_unif_dim_10_type_1.txt};
	\addplot[only marks, black, mark=*, mark options={black}, error bars/.cd,y dir=both,  y explicit] table [x index=0, y index=1, y error index=2]{pgfplots_data/results_onerdsa_unif_dim_10_type_1.txt};
	
	\addplot[clip marker paths=true,color = cyan,mark = none,thick,legend entry=1SPSA,] table [x index=0, y index=1] {pgfplots_data/results_onespsa_dim_10_type_1.txt};
	\addplot[only marks, cyan, mark=*, mark options={cyan}, error bars/.cd,y dir=both,  y explicit] table [x index=0, y index=1, y error index=2]{pgfplots_data/results_onespsa_dim_10_type_1.txt};

	\end{axis} 
	\end{tikzpicture}
	\caption{Evolution of the parameter error as the simulation budget is varied, for the first-order algorithms under the quadratic objective with $ N=10 $ and $ \sigma = 0.001 $.}
	\label{fig:first-order-line}
\end{figure}

Figure \ref{fig:first-order-line} compares the parameter error of 1RDSA-Perm-DP, 1RDSA-KW-DP, and both variants of 1RDSA and 1SPSA algorithms for the quadratic objective with dimension $10$, $ \sigma = 0.001 $ and a simulation budget of $50000$ function measurements. As in the case of the problem with dimension $5$, 1RDSA-Perm-DP and 1RDSA-KW-DP performed best, while the result of 1RDSA-Lex-DP is not reported due to its high inner loop length.

\subsection{Example 2: Fourth-order objective}
The function given below has been used for evaluating both RDSA and SPSA algorithms in \cite{prashanth2017rdsa} and \cite{spall2000adaptive}, respectively.
\begin{align}
	F(x,\xi) = x^TA^TAx +  0.1\sum_{j=1}^{N}(Ax)^3_j + 0.01\sum_{j=1}^{N}(Ax)^4_j + \xi ,
	\label{eq:fourth-order-fun}
\end{align}

\begin{figure}[H]
	\centering
\scalebox{.96}{
		\begin{tikzpicture}[
		declare function = {
			Z(\x,\y) = 0.25*\x^2 + 0.5*\x*\y + 0.5*\y^2 + 0.1*((0.5*\x + 0.5*\y)^3 + (0.5*\y)^3) + 0.01*((0.5*\x + 0.5*\y)^4 + (0.5*\y)^4);
		}
		]
		\begin{axis}
		[
		domain=-2:2,
		samples=40, 
		trig format plots=rad,
		xlabel=$x$,ylabel=$y$,
		enlargelimits=false,
		3d box=complete,
		grid,
		grid style={dashed,gray!40},
		axis line style={gray!40},
		colormap/bluered,
		]
		\addplot3 [surf] {Z(x,y)};
		\end{axis}
		\end{tikzpicture}}
		\caption{A plot of the fourth-order objective \eqref{eq:fourth-order-fun}, $ N=2 $.}
		\label{fig:fourth-order}
	\end{figure}
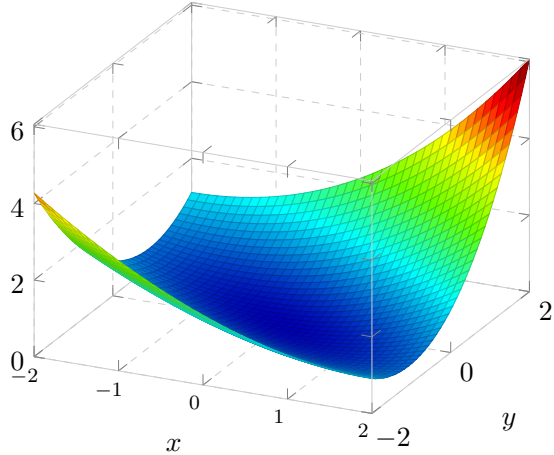

\noindent
where $ A $ and $ \xi $ are the same as in the quadratic objective. The initial point $ x_0 $ is set to the $ N $-dimensional vector of ones and the optimal point $ x^* $ is the $ N $-dimensional vector of zeros, with $f(x^*) = \E_{\xi}[ F(x^*,\xi)]= 0$. Figure \ref{fig:fourth-order} shows a plot of the fourth-order objective \eqref{eq:fourth-order-fun}.

Figure \ref{fig:second-order-bar} presents the parameter error for both first and second order algorithms in the case of the fourth-order objective \eqref{eq:fourth-order-fun} with dimension 5 and  $\sigma = 0.001 $ and $0.1$. Among the first-order algorithms, for both settings of noise, all algorithms except 1RDSA-Lex-DP exhibited similar performance. 

Among the second-order algorithms, similar to the quadratic case, for both settings of noise, 2RDSA-Perm-DP and 2RDSA-Lex-DP gave consistent performance with parameter error of the order of $ 10^{-2} $. Under the low noise setting, i.e., $\sigma=0.001$, 2RDSA-Unif exhibited the best performance, while under the high noise setting with $\sigma=0.1$, 2RDSA-Lex-DP outperformed the other algorithms. Thus, we observe that 2RDSA-Lex-DP and 2RDSA-Perm-DP algorithms are more tolerant to the noise, as compared to their random counterparts 2RDSA-AsymBer, 2RDSA-Unif, and 2SPSA.

Since the fourth-order objective is more difficult to optimize than the quadratic one, we observe that under the low noise setting, the parameter error in the case of the fourth-order objective for the first-order algorithms is higher, of the order of $ 10^{-2} $ compared to $ 10^{-5} $ for the quadratic case. For the second-order algorithms, the same is also of the order of $ 10^{-2} $ compared to $ 10^{-6} $ for the quadratic case. A similar trend is observed in the  high noise regime.

Figure \ref{fig:second-order-line} compares the parameter error of 2RDSA-Perm-DP, as well as both variants of 2RDSA and 2SPSA algorithms, for the fourth-order objective function with dimension $10$, $ \sigma = 0.001 $ and a simulation budget of $50000$ function measurements. As in the five-dimensional problem, 2RDSA-Unif and 2SPSA exhibited similar performance and outperformed the other algorithms. The results of the 2RDSA-Lex-DP algorithm are not displayed as it requires $ 3 \times 3^{10} $ function measurements per iteration.

\begin{figure*}[t]
	\centering
	\hspace{-2em}
	\begin{minipage}{\textwidth}
		\begin{tabular}{cc}
			\begin{subfigure}{0.5\textwidth}
				\centering
				\label{fig:first-order-fourth}
				\tabl{c}{\scalebox{0.95}{
						
						\begin{tikzpicture}
						\begin{axis}[
						ylabel=Parameter error,
						enlargelimits=0.15,
						symbolic x coords={1RDSA-Perm-DP,1RDSA-KW-DP,1RDSA-Lex-DP,1RDSA-AsymBer,1RDSA-Unif,1SPSA},
						xtick=data,
						x tick label style={rotate=45,anchor=east},
						ybar=5pt,
						bar width=9pt,
						legend pos=north west,
							y tick label style={
							/pgf/number format/.cd,
							fixed,
							fixed zerofill,
							precision=2,
							/tikz/.cd
						},
						]

						
							\addplot
						table[x=savefinal1,y=savefinal2,col sep=comma] {pgfplots_data/first_order_fourth_sigma_0.001.dat}; 
						
						\addplot
						table[x=savefinal1,y=savefinal2,col sep=comma] {pgfplots_data/first_order_fourth_sigma_0.1.dat}; 
						
						\legend{$ \sigma = 0.001 $, $\sigma = 0.1$}
						
						\end{axis}
						\end{tikzpicture}
				}}
				\caption{First-order algorithms}
			\end{subfigure}
			&
			\begin{subfigure}{0.5\textwidth}
				\centering
				\label{fig:second-order-fourth}
				\tabl{c}{\scalebox{0.95}{

						\begin{tikzpicture}
						\begin{axis}[
						ylabel=Parameter error, 
						enlargelimits=0.15,
						symbolic x coords={2RDSA-Perm-DP,2RDSA-Lex-DP,2RDSA-AsymBer,2RDSA-Unif,2SPSA},
						xtick=data,
						x tick label style={rotate=45,anchor=east},
						ybar=5pt,
						bar width=9pt,
						legend pos=north west,
						]

						\addplot
					table[x=savefinal1,y=savefinal2,col sep=comma] {pgfplots_data/second_order_fourth_sigma_0.001.dat}; 
					
					\addplot
					table[x=savefinal1,y=savefinal2,col sep=comma] {pgfplots_data/second_order_fourth_sigma_0.1.dat}; 
					
					\legend{$\sigma = 0.001$, $\sigma = 0.1$}

						\end{axis}
						\end{tikzpicture}
						
				}}
				\caption{Second-order algorithms}
			\end{subfigure}
		\end{tabular}
		\caption{Parameter error for various algorithms under the fourth-order objective \eqref{eq:fourth-order-fun} for a five-dimensional problem with a simulation budget of 50000 and $\sigma = 0.001 $ and $0.1$.}
		\label{fig:second-order-bar}
	\end{minipage}
\end{figure*}
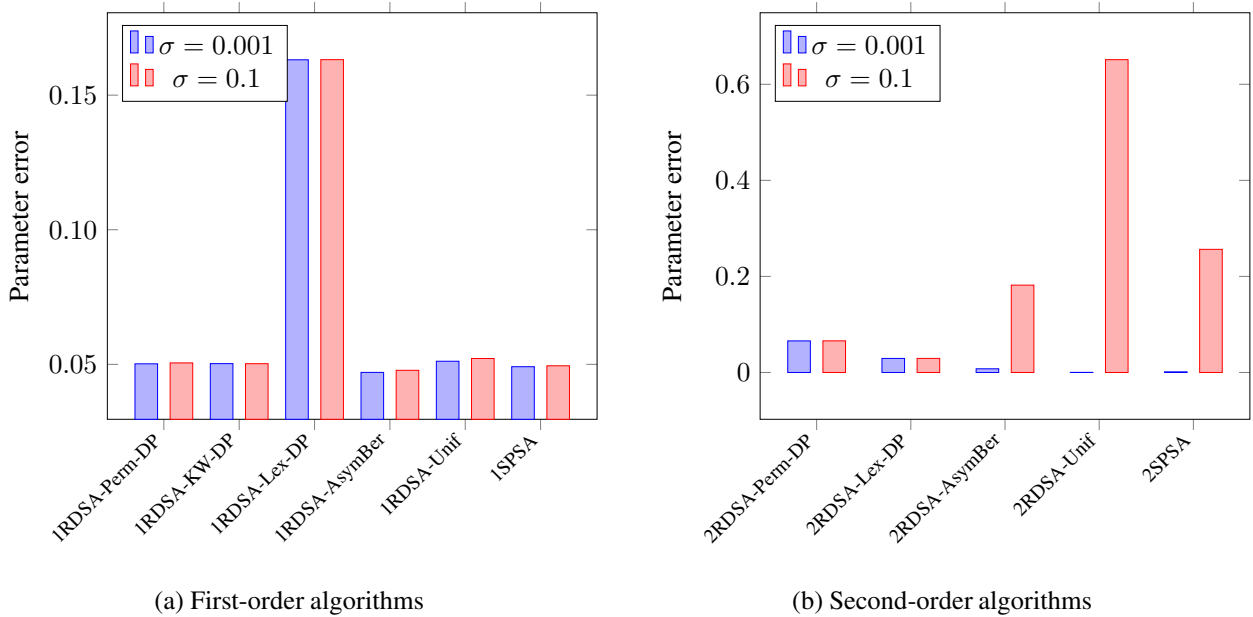

\begin{figure}[H]
		\centering
		\pgfplotsset{
			legend entry/.initial=,
			every axis plot post/.code={%
				\pgfkeysgetvalue{/pgfplots/legend entry}\tempValue
				\ifx\tempValue\empty
				\pgfkeysalso{/pgfplots/forget plot}%
				\else
				\expandafter\addlegendentry\expandafter{\tempValue}%
				\fi
			},
		}
		\scalebox{1}{
	\begin{tikzpicture} 
	\begin{axis}[
	ylabel=Parameter error,
	xlabel = \# Function measurements,
	]
	\addplot[clip marker paths=true,color = blue,mark = none,thick,legend entry=2RDSA-Perm-DP,] table [x index=0, y index=1] {pgfplots_data/results_twordsa_perm_dp_dim_10_type_2.txt};
	\addplot[only marks, blue, mark=*, mark options={blue}, error bars/.cd,y dir=both,  y explicit] table [x index=0, y index=1, y error index=2]{pgfplots_data/results_twordsa_perm_dp_dim_10_type_2.txt};
	
	\addplot[clip marker paths=true,color = red,mark = none,thick,legend entry=2RDSA-AsymBer,] table [x index=0, y index=1] {pgfplots_data/results_twordsa_asymber_dim_10_type_2.txt};
	\addplot[only marks, red, mark=*, mark options={red}, error bars/.cd,y dir=both,  y explicit] table [x index=0, y index=1, y error index=2]{pgfplots_data/results_twordsa_asymber_dim_10_type_2.txt};

	\addplot[clip marker paths=true,color = black,mark = none,thick,legend entry=2RDSA-Unif,] table [x index=0, y index=1] {pgfplots_data/results_twordsa_unif_dim_10_type_2.txt};
	\addplot[only marks, black, mark=*, mark options={black}, error bars/.cd,y dir=both,  y explicit] table [x index=0, y index=1, y error index=2]{pgfplots_data/results_twordsa_unif_dim_10_type_2.txt};
	
	\addplot[clip marker paths=true,color = green,mark = none,thick,legend entry=2SPSA,] table [x index=0, y index=1] {pgfplots_data/results_twospsa_dim_10_type_2.txt};
	\addplot[only marks, green, mark=*, mark options={green}, error bars/.cd,y dir=both,  y explicit] table [x index=0, y index=1, y error index=2]{pgfplots_data/results_twospsa_dim_10_type_2.txt};
	
	\end{axis} 
	\end{tikzpicture}}
	\caption{Evolution of the parameter error as the simulation budget is varied, for the second-order algorithms under the fourth-order objective with $ N=10 $ and $\sigma = 0.001$.}
	\label{fig:second-order-line}
\end{figure}
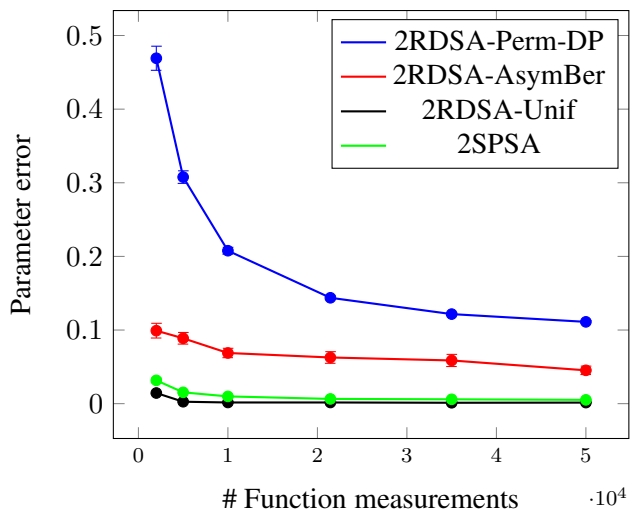

\subsection{Example 3: Rastrigin objective}

The Rastrigin objective function is defined as follows:
\begin{align}
	F(x,\xi) = \sum_{i=1}^{N}(x_i^2 - 10\cos(2\pi x_i))  + 10N + 1 + \xi,
	\label{eq:Rastrigin-fun}
\end{align}

\noindent
where $ \xi $ is the same as for the quadratic objective. The initial point $ x_0 $ is set to the $ N $-dimensional vector of twos and the optimal point $ x^* $ is the $ N $-dimensional vector of zeros, with $f(x^*) = \E_{\xi}[ F(x^*,\xi)]= 1$. Figure \ref{fig:rastrigin} shows a plot of the Rastrigin objective \eqref{eq:Rastrigin-fun}, which has many local minima.

\begin{figure*}[t]
	\centering
	\hspace{-2em}
	\begin{minipage}{\textwidth}
		\begin{tabular}{cc}
			\begin{subfigure}{0.5\textwidth}
				\centering
				\label{fig:first-order-Rastrigin}
				\tabl{c}{\scalebox{0.95}{
						
						\begin{tikzpicture}
						\begin{axis}[
						ylabel=Parameter error,
						enlargelimits=0.15,
						symbolic x coords={1RDSA-Perm-DP,1RDSA-KW-DP,1RDSA-Lex-DP,1RDSA-AsymBer,1RDSA-Unif,1SPSA},
						xtick=data,
						x tick label style={rotate=45,anchor=east},
						ybar=5pt,
						bar width=9pt,
						legend pos=north west,
						]
						
						\addplot
						table[x=savefinal1,y=savefinal2,col sep=comma] {pgfplots_data/first_order_rastrigin_sigma_0.001.dat}; 
						
						\addplot
						table[x=savefinal1,y=savefinal2,col sep=comma] {pgfplots_data/first_order_rastrigin_sigma_0.1.dat}; 
						
						\legend{$ \sigma = 0.001 $, $\sigma = 0.1$}
						
						
						\end{axis}
						\end{tikzpicture}
					}}
					\caption{First-order algorithms}
				\end{subfigure}
				&
				\begin{subfigure}{0.5\textwidth}
					\centering
					\label{fig:second-order-Rastrigin}
					\tabl{c}{\scalebox{0.95}{

							\begin{tikzpicture}
							\begin{axis}[
							ylabel=Parameter error,
							enlargelimits=0.15,
							symbolic x coords={2RDSA-Perm-DP,2RDSA-Lex-DP,2RDSA-AsymBer,2RDSA-Unif,2SPSA},
							xtick=data,
							x tick label style={rotate=45,anchor=east},
							ybar=5pt,
							bar width=9pt,
							legend pos=north west,
							]
							
							\addplot
							table[x=savefinal1,y=savefinal2,col sep=comma] {pgfplots_data/second_order_rastrigin_sigma_0.001.dat}; 
							
							\addplot
							table[x=savefinal1,y=savefinal2,col sep=comma] {pgfplots_data/second_order_rastrigin_sigma_0.1.dat}; 
							
							\legend{$ \sigma = 0.001 $, $\sigma = 0.1$}
							
							\end{axis}
							\end{tikzpicture}
							
						}}
						\caption{Second-order algorithms}
					\end{subfigure}
				\end{tabular}
				\caption{Parameter error for various algorithms under the Rastrigin objective \eqref{eq:Rastrigin-fun} for a five-dimensional problem and a simulation budget of 50000 and $\sigma = 0.001 $ and $0.1$.}
				\label{fig:Rastrigin-bar}
			\end{minipage}
		\end{figure*}
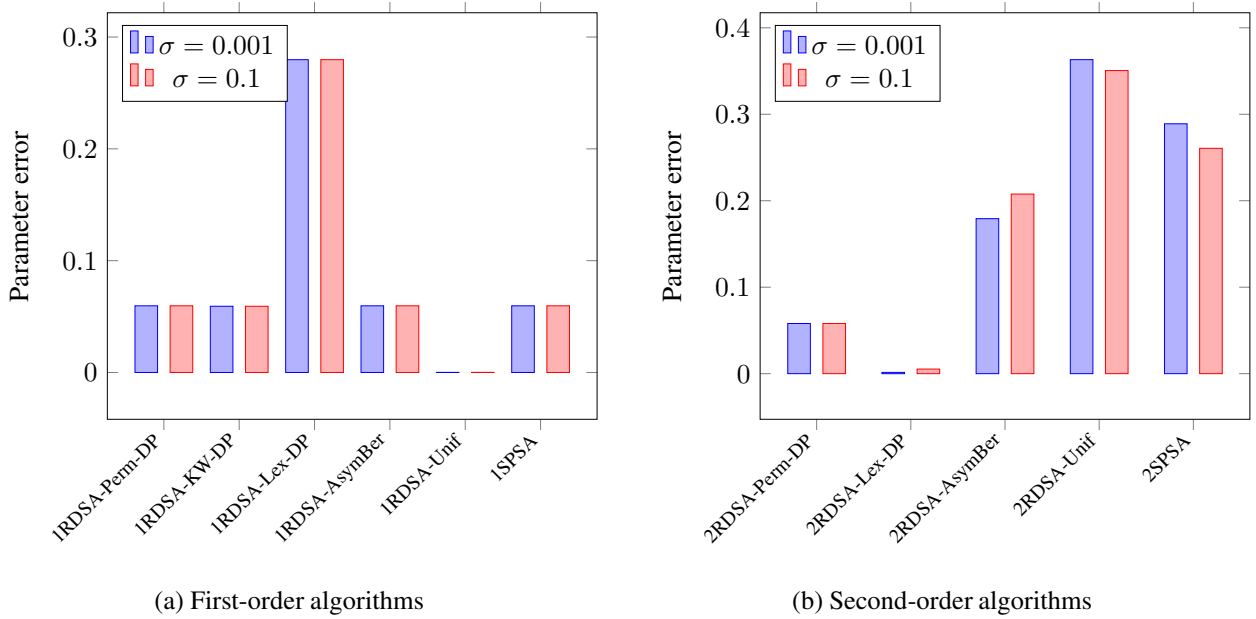
		
Figure \ref{fig:Rastrigin-bar} presents the parameter error for both first-order and second-order algorithms for the Rastrigin objective \eqref{eq:Rastrigin-fun} with dimension 5 and $\sigma = 0.001 $ and $0.1$.

For both settings of noise, among the first-order algorithms, 1RDSA-Unif outperformed the other algorithms, while 2RDSA-Lex-DP performed the best among the second-order algorithms.

Similar to the quadratic and fourth-order objective, first-order algorithms, 2RDSA-Perm-DP and 2RDSA-Lex-DP gave consistent performance under both settings of noise. 

In summary, for all the three objectives, among the first-order algorithms, we observed that 1RDSA-Perm-DP and 1RDSA-KW-DP performed best, while 1RDSA-Lex-DP showed poor performance. On the other hand, among the second-order algorithms, 2RDSA-Lex-DP exhibited the best performance.

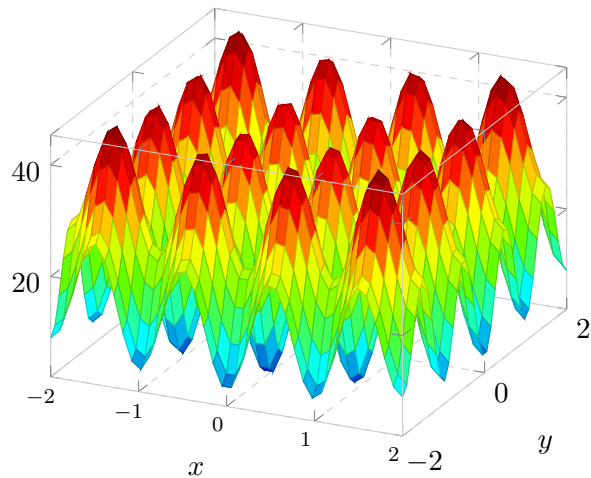
\begin{figure}[H]
	\centering
	\begin{tikzpicture}[
	declare function = {
		Z(\x,\y) = \x^2 - 10*cos(2*pi*\x) + \y^2 - 10*cos(2*pi*\y) + 20 + 1;
	}
	]
	\begin{axis}
	[
	domain=-2:2,
	samples=40, 
	trig format plots=rad,
	xlabel=$x$,ylabel=$y$,
	enlargelimits=false,
	3d box=complete,
	grid,
	grid style={dashed,gray!40},
	axis line style={gray!40},
	colormap/bluered,
	]
	\addplot3 [surf] {Z(x,y)};
	\end{axis}
	\end{tikzpicture}
	\caption{A plot of the Rastrigin objective \eqref{eq:Rastrigin-fun}, $ N=2 $.}
	\label{fig:rastrigin}
\end{figure}

\section{Conclusions and Future Work}
\label{sec:concl}
We incorporated two novel deterministic perturbation (DP) schemes into the RDSA class of simultaneous perturbation algorithms. The proposed DP variants of first-order, as well as second-order, RDSA were shown to result in asymptotically unbiased gradient/Hessian estimates, thus resulting in provably convergent 1RDSA/2RDSA variants. We also performed numerical experiments to validate the theoretical findings. 

As future work, it would be interesting to perform a finite-time analysis of the RDSA schemes and study the impact of deterministic perturbations therein. To the best of our knowledge, non-asymptotic bounds are not available for vanilla RDSA algorithms, especially. in a general simulation optimization setting where convexity cannot be assumed. Another fruitful direction would be try the deterministic perturbation variants of RDSA in sophisticated applications, e.g., in transportation, networks and service systems.

\bibliographystyle{plainnat}
\bibliography{references}

\end{document}